\documentclass{cmip-l}

\usepackage{amssymb,amsmath,amsthm,amscd,mathrsfs,graphicx}
\usepackage[cmtip,all]{xy}
\usepackage{amscd}
\usepackage{url}

\theoremstyle{plain}
\newtheorem{theorem}{Theorem}[subsection]
\newtheorem{corollary}[theorem]{Corollary}

\newtheorem{lemma}[theorem]{Lemma}
\newtheorem{proposition}[theorem]{Proposition}
\theoremstyle{definition}
\newtheorem{definition}[theorem]{Definition}
\newtheorem{remark}[theorem]{Remark}
\newtheorem{example}[theorem]{Example}

\newcommand{\Hilb}{\operatorname{Hilb}}

\newcommand{\Quot}{\operatorname{Quot}}

\newcommand{\bbC}{\mathbf{C}}

\newcommand{\bbG}{\mathbf{G}}

\newcommand{\bbP}{\mathbf{P}}

\newcommand{\bbZ}{\mathbf{Z}}

\newcommand{\calF}{{ \mathcal F}}

\newcommand{\calK}{{ \mathcal K}}
\newcommand{\calL}{{ \mathcal L}}
\newcommand{\calM}{{ \mathcal M}}

\newcommand{\calO}{{ \mathcal O}}

\newcommand{\Hom}{\operatorname{Hom}}
\newcommand{\Ext}{\operatorname{Ext}}

\newcommand{\ShHom}{\underline{\operatorname{Hom}}}
\newcommand{\ShExt}{\underline{\operatorname{Ext}}}

\newcommand{\Spec}{\operatorname{Spec}}
\begin{document}

\title{Singular curves and their compactified Jacobians}


\author{Jesse Leo Kass}
\address{Institut f\"{u}r Algebraische Geometrie, Leibniz Universit\"{a}t Hannover, Welfengarten 1, D-30167 Hannover, Germany}
\email{kass@math.harvard.edu}
\thanks{}



\date{}

\begin{abstract}
We survey the theory of the compactified Jacobian associated to a singular curve.  We focus on describing low genus examples using the Abel map.
\end{abstract}

\maketitle

In this article we study how to assign a degenerate Jacobian, called a compactified Jacobian, to a singular curve.  The title of this article is intended to recall Mumford's book ``Curves and their Jacobians" \cite{mumford75}.  That book  contains a beautiful introduction to the geometric theory of the Jacobian variety associated to a smooth curve, and the present article is intended to be a survey of the analogous theory for a singular curve, written in a similar spirit.  The focus is on providing examples and indicating various technical issues. We omit many proofs and instead direct the interested reader to the literature.

The main goal of this article is to show how the Abel map can be used to describe the compactified Jacobian $\bar{J}_{X}^{d}$ associated to a singular curve $X$.  One novel feature of this article is that we link the theory of the Abel map to the theory of linear systems of generalized divisors.  Such a link is certainly well-known to experts (see \cite[Rmk.~1.6.5]{hartshorne86}), but we develop the relation in greater depth than has previously been done in the literature.  We also discuss several examples that have not appeared in the literature before.  The most interesting example is the compactified Jacobian of a genus $2$ non-Gorenstein curve, which is studied throughout the paper but especially in Example~\ref{Example: AbelForSingularGenus2}.

Most of the results in this paper are not due to the author.  Many mathematicians have contributed to the body of work discussed, but the  author owes a particularly large mathematical debt to Allen Altman, Steve Kleiman, and Robin Hartshorne.  The theory of generalized divisors was developed by Hartshorne in \cite{hartshorne86}, \cite{hartshorne94}, and \cite{hartshorne07}, and the Abel map that we study here was constructed by Altman--Kleiman in \cite{altman80}. Kleiman's article \cite{kleiman05} was also very helpful in writing Section~\ref{Sect: Jacobian}.

The reader who has looked below at the ``Conventions" section may have noticed that in this paper the term ``curve" always refers to an irreducible and reduced curve.  It is possible to assign a compactified Jacobian $\bar{J}_{X}^{d}$ to, say, a reducible curve $X$, but both the definition of $\bar{J}_{X}^{d}$ and the associated Abel map then become more complicated.  One major barrier to constructing an Abel map for a reducible curve is that the theory of linear systems on a reducible curve has undesirable properties, which are discussed in \cite[Rmk.~2.9]{hartshorne07}.  There is, however, a growing body of work (e.g.~\cite{caporaso07, coelho08, coelho10}) on constructing an Abel map for a reducible but reduced curve.  The papers just cited also provide a guide to the literature on compactified Jacobians of reducible curves.  Much less is known about assigning a compactified Jacobian to a non-reduced curve, and in particular, there seems to be no literature on constructing an Abel map for such a curve.

\subsection{Organization}
This paper is organized as follows.  In Section~\ref{Sect: Jacobian} we recall some basic facts about the Jacobian variety associated to a smooth curve.  The material developed there is our model for the theory of the compactified Jacobian associated to a singular curve.  We begin discussing singular curves in Section~\ref{Sec: GeneralizedJac}, where we define the generalized Jacobian of a singular curve.  This variety typically fails to be proper, and so we are naturally led to compactify the generalized Jacobian.  This is done in Section~\ref{Sect: CompJac}.  At the end of that section we show by example that the most naive approach to constructing an Abel map into the compactified Jacobian fails.  The rest of the article is devoted to constructing a suitable Abel map.  We recall the theory of generalized divisors in Section~\ref{Section: GeneralizedDivisors} and then we apply that theory to construct the Abel map in Section~\ref{Section: Abel Map}.  Finally, Section~\ref{Section: Dualizing} is an appendix that contains some facts about the dualizing sheaf that are used in this article.

\subsection{Personal Remarks}
A few personal words about this article.  I wrote this article for the proceedings for the conference ``The View From Joe's Office," held in honor of Joe Harris' 60th birthday.  Joe was my adviser in graduate school, and I hope that this article demonstrates Joe's influence on me as a mathematician.  

During my last year in graduate school, William Fulton told me about working with Joe when Joe first moved to Brown University.  Fall semester that year Joe taught a topics course on Brill--Noether theory.  After reviewing the necessary definitions, Joe begin by working  out the theory of special divisors on a genus $2$ curve.  The next day of class was spent on genus $3$ curves, which took up one lecture, and Joe then moved on to genus $4$ curves, which took a bit more class time.  Fulton said that he expected Joe to soon run out of examples and then state and prove the general Brill--Noether Theorem.  

Joe continued doing examples until the Thanksgiving break.

Right before the break, he was discussing curves whose genera was in the double digits.  When he returned from break, he apologized and explained that he had run out of examples. It was only then that he stated and proved the Brill--Noether Theorem.  Fulton cited this as an example of Joe's excellent mathematical taste: explicitly working through such a large class of curves provided incredible insight into Brill--Noether theory, insight that is not conveyed by just proving general theorems.

In the present article, we will certainly not get to singular curves with double digit genera, but I hope the choice of material shows the influence of Joe's good taste.  We will work out examples of compactified Jacobians associated to curves of low genus, and the selection of examples was influenced by \cite{ACGH},  a book that Joe co-authored.  Indeed, the examples in Section~\ref{Sect: Jacobian} are all answers to exercises in \cite{ACGH}, and the examples studied later are chosen as examples analogous to the ones from Section~\ref{Sect: Jacobian}.

\section*{Conventions}
\begin{itemize}
	\item The letter $k$ denotes an algebraically closed field.
	
	\item If $V$ is a $k$-vector space, then $\bbP V$ is the Grassmannian of $1$-dimensional quotients of $V$.

	\item If $T$ is a $k$-scheme, then we write $X_{T}$ for $X \times_{k} T$.

	\item If $T$ and $X$ are $k$-schemes, then a \textbf{$T$-valued point} of $X$ is a $k$-morphism $T \to X$.

	\item A \textbf{variety} is a finite type, separated, and integral $k$-scheme.

	\item A \textbf{curve} is a finite type, separated, integral, and projective $k$-scheme of pure dimension $1$.
	\item The \textbf{genus} $g$ of a curve $X$ is $g:=1-\chi(\calO_{X})$.
	\item The symbol $k(X)$ denotes the \textbf{field of fractions} (or field of rational functions) of a curve $X$.
	\item The symbol $\calK$ denotes the locally constant sheaf associated to $k(X)$.
	\item The symbol $\calK_{\omega}$ denotes the locally constant sheaf of rational $1$-forms.
	\item If $X$ is a $k$-scheme and $F$, $G$ are two $\calO_{X}$-modules, then we write $\ShHom(F,G)$ for the sheaf of homomorphisms from $F$ to $G$.
\end{itemize}

\section{The Jacobian} \label{Sect: Jacobian}
Here we discuss the Jacobian variety $J_{X}^{d}$ of a smooth curve $X$.  We begin by discussing three different approaches to constructing $J_{X}^{d}$.  Of these approaches, one uses the Abel map, and we also review the properites of this map.  After recalling the definition, we conclude this section by describing the Abel map of a curve of genus at most $4$.  In the subsequent sections of this article we will work to extend results of this section from smooth curves to singular curves.

The most succinct definition of the Jacobian $J^{0}_{X}$ of a smooth curve $X$ requires us to assume that the ground field $k$ is the field of complex numbers $k = \bbC$.  If $X$ is a smooth projective complex curve, then the associated Jacobian is the complex torus
\begin{displaymath}
	J^{0}_{X} := H^{1}(X_{\text{an}}, \calO_{X})/H^{1}(X_{\text{an}}, \underline{\bbZ}).
\end{displaymath}
We write $X_{\text{an}}$ for the space $X$ with the analytic or classical topology (rather than the Zariski topology).  The group $H^{1}(X_{\text{an}}, \underline{\bbZ})$ is considered as a subgroup of $H^{1}(X_{\text{an}}, \calO_{X})$ using the map induced by the natural inclusion $\underline{\bbZ} \stackrel{2 \pi i }{\longrightarrow}  \calO_{X}$.  This inclusion is part of a short exact sequence 
\begin{equation} \label{Eqn: LineBundleAsCoSES }
	0 \to  \underline{\bbZ} \to \calO_{X} \to \calO_{X}^{\ast} \to 0,
\end{equation}
which provides us with an alternative description of $J_{X}^{0}$.  An inspection of the associated long exact sequence shows that there is a canonical isomorphism
\begin{equation} \label{Eqn: LineBundlesAsCoh}
	J_{X}^{0} \cong \ker( H^{1}(X_{\text{an}}, \calO_{X}^{\ast}) \stackrel{c_1}{\longrightarrow} H^{2}(X_{\text{an}}, \underline{\bbZ})),
\end{equation} 
where $c_{1}$ is the 1st Chern class map.

The group $H^{1}(X_{\text{an}}, \calO_{X}^{\ast})$ is canonically isomorphic to the set of isomorphism classes of line bundles, and under this isomorphism, the 1st Chern class map corresponds to the degree map, so the points of $J_{X}^{0}$ are in natural bijection with the degree $0$ line bundles on $X$.  

This description suggests a way of defining $J_{X}^{0}$ over an arbitrary ground field: $J_{X}^{0}$ is the moduli space of degree $0$ line bundles. In slightly more generality, if $d$ is any integer, then we define the degree $d$ Jacobian, or the moduli space of line bundles of degree $d$, as follows.
\begin{definition} \label{Def: Jacobian}
	The \textbf{Jacobian functor} $J_{X}^{d,\sharp} \colon $k$\text{-Sch} \to \text{Set}$ of degree $d$ is the \'{e}tale sheaf associated to the functor that assigns to a $k$-scheme $T$ the set of isomorphism classes of line bundles $L$ on $X_{T}$ that have the property that the restriction of $L$ to every fiber of $X_{T} \to T$ has degree $d$.  The \textbf{Jacobian variety} $J^{d}_{X}$ of degree $d$  is the $k$-scheme $J_{X}$ that represents $J^{d, \sharp}_{X}$.
\end{definition}
Our definition of $J^{d, \sharp}_{X}$ is somewhat unsatisfactory.  Because $J^{d, \sharp}_{X}$ is the sheafification of the functor parameterizing isomorphism classes of line bundles and not the functor itself, it is not immediately clear what the functor  $J^{d, \sharp}_{X}$ parametrizes.  For example, a line bundle $L$ on $X_{T}$ defines a $T$-valued point of $J^{d, \sharp}_{X}$ (for a suitable $d$, provided the degree of the restriction of $L$ to the fiber of $X_{T} \to T$ over $t \in T$ is constant as a function of $t$), but it is not immediate from the definition that  every $T$-valued point of $J^{d, \sharp}_{X}$ is defined by some $L$.  Similarly, it is unclear when two line bundles $L$ and $M$ on $X_{T}$ define the same $T$-valued point.

A very careful discussion of this topic can be found in \cite{kleiman05}.  In this article, the distinction between the functor parameterizing line bundles on $X$ and its associated sheaf will be largely irrelevant for the following reason: when $T$ equals $\Spec(K)$ for $K$  an algebraically closed field, $J_{X}^{d, \sharp}(T)$ equals the set of isomorphisms classes of degree $d$ line bundles on $X_{T}$, which is what one should naively expect.  (This is \cite[Ex.~2.3]{kleiman05}.) 

In an arithmetic context, however, it is important to distinguish between the functor parameterizing line bundles and its sheafification because when $K$ fails to be algebraically closed there may be $K$-valued points of $J_{X}^{d, \sharp}$ that cannot be represented by line bundles.  (See \cite[Ex.~2.4]{kleiman05} for an example.)

In any case, to make use of Definition~\ref{Def: Jacobian}, we need to prove that $J_{X}^{d, \sharp}$ can be represented by a $k$-scheme.  In general, a standard approach to 
proving representability of a functor is to use a  theorem of Artin.   In \cite{artin74, artin69}, Artin gave  a  criteria for a functor to be representable by an algebraic space, and 
one approach to proving that a functor can be represented by a scheme is to first prove representability by an algebraic space by verifying Artin's criteria and then to prove that the resulting algebraic space is actually a scheme by some other means.

In the specific case of $J^{d, \sharp}_{X}$, this strategy was carried out by Artin in \cite{artin69}.  He first proved that $J^{d, \sharp}_{X}$ is representable by an algebraic space by verifying that the functor satisfies Artin's criteria (\cite[Thm.~7.3, p.~67]{artin69}; see also \cite[p.186-187]{artin74}).  Artin then proved that this algebraic space is a $k$-scheme by proving more generally that any torsor for a locally finite type group space over $k$ is actually a scheme  \cite[Lem.~4.2, p.~43]{artin69}.  

A second approach to proving that $J^{d, \sharp}_{X}$ is representable is to make use of the Abel map.  While this approach to proving representability is not as general as the first approach, it has the advantage of providing more insight into the geometry of $J^{d}_{X}$.  Let us review this construction, beginning with the definition of the symmetric power.

\begin{definition} \label{Def: SymmetricPower}
	The \textbf{symmetric product} $X^{(d)}$ is defined to be the quotient of the $d$-fold self-product $X \times \dots \times X$ of $X$ by the action of the symmetric group $\operatorname{Sym}_{d}$ given by permuting factors.
\end{definition}
Recall that the quotient of a quasi-projective variety $V$ by a finite group always exists as a quasi-projective variety (by, say, \cite[p.~66]{mumford70}).  Hence the symmetric power $X^{(d)}$ of a smooth curve is a projective variety of dimension $d$.  This symmetric power is also smooth, but this is more difficult to establish.  A proof of smoothness can be found in e.g.~\cite[Thm.~7.2.3]{fantechi05}.  The symmetric power of a smooth curve also has a moduli-theoretic interpretation.

\begin{definition} \label{Def: HilbFunctor}
	The \textbf{Hilbert functor} $\Hilb^{d, \sharp}_{X}$ of degree $d$ is defined by setting $\Hilb^{d, \sharp}(T)$ equal to the set of $T$-flat closed subschemes $Z \subset X_{T}$ with the property that every fiber of $Z \to T$ has degree $d$. 
\end{definition}

\begin{lemma} \label{Lemma: SymIsHilb}
	The $k$-scheme $X^{(d)}$ represents $\Hilb_{X}^{d, \sharp}$.
\end{lemma}
\begin{proof}
	We will prove this lemma under the additional assumption that $\Hilb^{d, \sharp}_{X}$ can be represented by some $k$-scheme $\Hilb^{d}_{X}$.  To begin, let us construct a natural transformation
	$X^{(d)} \to \Hilb_{X}^{d, \sharp}$.  It is enough to construct a $\operatorname{Sym}_{d}$-invariant transformation $X \times \dots \times X \to \Hilb^{d, \sharp}_{X}$, and we construct this second transformation by exhibiting the corresponding closed subscheme $Z \subset (X \times \dots \times X) \times X$.
	
	We define $Z$ to be the multi-diagonal that consists of tuples $(p_1, \dots, p_d, q)$ with $p_i = q$ for some $i$.  Working on the level of local rings, one can show that $Z \to X \times \dots \times X$ is flat and finite of degree $d$.  Because $Z$ is visibly $\operatorname{Sym}_{d}$-invariant, this subscheme induces the desired transformation $X^{(d)} \to \Hilb^{d, \sharp}_{X}$.  To complete the proof, we need to show that this transformation is an isomorphism.
	
	First, observe that $X^{(d)}(K) \to \Hilb^{d, \sharp}_{X}(K)$ is injective for any algebraically closed field $K$.  Indeed, this observation is just the statement that no two fibers of $X \times X^{(d)} \supset Z \to X^{(d)}$ are equal as subschemes, a statement that can be  verified by working affine locally and writing out the ideal of $Z$ in terms of symmetric polynomials.  (Show that the fiber of  $Z \to X^{(d)}$ over the point that is the image of $(p_1, \dots, p_{d})$ is a closed subscheme supported on $\{ p_{1}, \dots, p_{d} \}$ and then compute the  length of $\calO_{Z, p_i}$.)
	
	Now a similar explicit computation shows that $X^{(d)} \to \Hilb^{d}_{X}$	 is an isomorphism over the locus in $\Hilb^{d}_{X}$ parameterizing reduced subschemes, so the morphism is birational.  The morphism is also quasi-finite because we have just shown that it is injective on $K$-points.  Thus $X^{(d)} \to \Hilb^{d}_{X}$ is a quasi-finite, birational morphism.  Furthermore, both the target and the source of the morphism are smooth and proper.  Indeed, we made this observation about $X^{(d)}$ just after defining the variety, and one can verify that $\operatorname{Hilb}^{d}_{X}$ is smooth and proper by using the functorial characterization of these properties.   We can conclude by Zariski's Main Theorem that  $X^{(d)} \to \Hilb^{d}_{X}$ is an isomorphism.
\end{proof}
By the lemma, we can identify $X^{(d)}$ with $\Hilb^{d, \sharp}_{X}$ when $X$ is smooth. The symmetric power is still defined when $X$ is singular, but then $X^{(d)}$ may not represent the Hilbert functor.  In Section~\ref{Section: Abel Map} we will describe how the symmetric power of a singular curve is related to the Hilbert functor.

The closed subschemes $Z \subset X_{T}$ that correspond to elements of $\Hilb^{d}_{X}(T)$ are all effective relative Cartier divisors.  That is, locally the ideal of $Z \subset X_{T}$ is generated by a single element.  Indeed, this is a consequence of the proof of Lemma~\ref{Lemma: SymIsHilb}, but the fact can also be proven directly (see \cite[Lem.~9.3.4]{kleiman05}).  Linear equivalence classes of Cartier divisors are in natural bijection with isomorphism classes of line bundles, and the Abel map $A \colon X^{(d)} \to J^{d}_{X}$ is a manifestation of this relation.  

Before defining the Abel map, let us review the definition of a Cartier divisor from \cite[Chap.~2, Sect.~5]{hartshorne77}.  We define a \textbf{Cartier divisor} $D$ to be a global section of the quotient sheaf $\calK^{\ast}/\calO_{X}^{\ast}$, where $\calK^{\ast}$ is the locally constant sheaf of nonzero rational functions.  Because Cartier divisors are sections of an abelian sheaf, we can form the \textbf{minus} $-D$ of a Cartier divisor $D$ and the \textbf{sum} $D+E$ of $D$ and a second Cartier divisor $E$.  As a section of a quotient sheaf, $D$  can be represented by a collection $(f_i,U_i)_{i \in I}$ consisting of open subsets $\{ U_i \}_{i \in I}$ that cover $X$ and rational functions $\{ f_i \}_{i \in I}$ that have the property that $f_i f_j^{-1}$ is regular on $U_i \cap U_j$.  To this data we can associate a nonzero coherent subsheaf
\begin{displaymath}
	I_{D} \subset \calK,
\end{displaymath}
namely the subsheaf generated by $f_i$ on $U_i$.  

The reader may verify that the correspondence $D \mapsto I_{D}$ defines a bijection between the set of global sections of $\calK^{\ast}/\calO_{X}^{\ast}$ and the set of nonzero coherent subsheaves $I_{D} \subset \calK$ that are locally principal.  (In fact, nonzero subsheaves of $\calK$ are always locally principal; this can be proven using, say, the classification of f.g.~modules over a DVR.)  Thus we can think of Cartier divisors as nonzero subsheaves of $\calK$ rather than as global sections of $\calK^{\ast}/\calO_{X}^{\ast}$.  We will work with Cartier divisors as subsheaves in this article.

A Cartier divisor $D$ is said to be \textbf{effective} if we have $I_{D} \subset \calO_{X}$. In other words, an effective Cartier divisor is a $0$-dimensional closed subscheme.  A rational function $f$ defines a Cartier divisor, the Cartier divisor $\calO_{X} \cdot f \subset \calK$, and we denote this divisor by $\operatorname{div}(f)$.  We say that two Cartier divisors $D$ and $E$ are \textbf{linearly equivalent} if $D = E + \operatorname{div}(f)$ for some rational function $f$.  The set of all effective divisors linearly equivalent to a given divisor $D$ is denoted by $|D|$ and called the \textbf{complete linear system} associated to $D$.

The divisors of the form $\operatorname{div}(f)$ naturally form a subgroup of the group of all Cartier divisors. In fact, the rule $D \mapsto I_{D}$ defines a surjective group homomorphism from the group of Cartier divisors to the group of isomorphism classes of line bundles, and the kernel of this map is precisely the subgroup consisting of divisors of the form $\operatorname{div}(f)$.  Thus a linear equivalence class of Cartier divisors is the same thing as an isomorphism class of line bundles. In slightly different form, this is  \cite[Prop.~6.13]{hartshorne77}.   

The complete linear system $|D|$ associated to a Cartier divisor naturally has the structure of a projective space.  Indeed, consider the $\calO_{X}$-linear dual  $\calL(D) :=\ShHom(I_{D}, \calO_{X})$ of $I_{D}$.  The natural map  $H^{0}(X, \calL(D)) \to \Hom(I_{D}, \calO_{X})$ is an isomorphism, and the rule that sends a nonzero global section of $\calL(D)$ to the image of the corresponding homomorphism $I_{D} \to \calO_{X}$  defines a bijection $\bbP H^{0}(X, \calL(D)) \cong |D|$.  We use  $\calL(D)$ to define the Abel map.
\begin{definition}
	If $D \in \Hilb^{d, \sharp}_{X}(T)$ (for some $k$-scheme $T$), then we define $\calL(D) := \ShHom(I_{D}, \calO_{X_{T}})$.  The $d$-th \textbf{Abel map} $A \colon X^{(d)} \to J^{d}_{X}$ is defined by the rule	$D \mapsto \calL(D)$.
\end{definition} 
Because $I_{D}$ is a line bundle, the formation of $\calL(D)$ commutes with base change, so the Abel map is well-defined. 

By definition, the fiber of $A$ over $[L] \in J^{d}_{X}$ is the set of divisors $D$ satisfying $\calL(D) \cong L$.    In other words, $A^{-1}([L]) = \bbP H^{0}(X, L)$.  This equality holds on the level of schemes, though this is perhaps not clear from the work we have done so far.  A proof of scheme-theoretic equality and other technical results describing the structure of $A$ can be found in \cite[Sect.~9.3]{kleiman05} (esp. Thm.~9.3.13).  In any case, it follows from those structural results that $A$ fibers $X^{(d)}$ over $J^{d}_{X}$ by projective spaces of possibly varying dimension.

What dimensions can these projective spaces have?  The dimension $\dim |D|$ is controlled by the Riemann--Roch Formula.  Define a \textbf{canonical divisor} $K$ to be a divisor with the property that $\calL(K) = \omega$, the dualizing sheaf of $X$.  Given a Cartier divisor $D$, an \textbf{adjoint divisor} $\operatorname{adj} D$ is defined to be a Cartier divisor satisfying $\calL(\operatorname{adj} D)=\ShHom( \calL(D), \omega)$ or equivalently a divisor linearly equivalent to $K-D$.  The Riemann--Roch Formula relates $\dim |\operatorname{adj} D|$ to $\dim |D|$.  The formula states
\begin{displaymath}
	\dim |D| - \dim |\operatorname{adj} D| = d+1-g,
\end{displaymath}
where $d$ is the degree of $D$.  

To use this formula, we need information about $\operatorname{adj} D = K-D$. The canonical divisor $K$ has degree $2g-2$, so if $D$ has degree $d>2g-2$, then $\operatorname{adj} D$ has negative degree, which forces $\dim |\operatorname{adj} D| = -1$.  To study $|\operatorname{adj} D|$ for $d \le 2g-2$, we introduce the canonical map. 

Assume $g \ge 1$.  (The case $g=0$ can be dealt with separately.)  The canonical divisor $K$ of a curve of genus $g \ge 1$ is base-point free (by \cite[Prop.~5.1]{hartshorne77}), and so $K$ determines a morphism $X \to \bbP H^{0}(X, \calL(K))^{\vee}$ to projective space that we call the \textbf{canonical map}.  The image is a curve when $g \ge 2$, and we call this curve the \textbf{canonical curve}.  This curve is a curve in $\bbP^{g-1}$ as  $\dim H^{0}(X, \calL(K)) = g$.  In terms of the canonical map, the adjoint linear system $|\operatorname{adj} D|$ associated to an effective divisor $D$ can be described as the set of hyperplanes in $\bbP^{g-1}$ whose preimage contains $D$.  Using this description and the Riemann--Roch Formula, one can prove that 
\begin{equation} \label{Eqn: GeneralLinearSystem}
	\dim |D|  = 	\begin{cases}
					-1 & \text{if $d <g$} \\
					d-g& \text{if $d \ge g$}
				\end{cases}
\end{equation}
for a general effective divisor $D$ of degree $d$ and for every divisor when $d \ge 2g-1$.  We give the proof later in Section~\ref{Section: GeneralizedDivisors}, where we prove a more general statement as Corollary~\ref{Cor: GorensteinGenLinSystem}.

For the purpose of constructing $J^{d}_{X}$, the most important fact about Cartier divisors is that $\dim |D|=d-g$ for every divisor of degree $d>2g-2$.  In other words, for such a degree, the fibers of $A \colon X^{(d)} \to J^{d, \sharp}_{X}$ are all $\bbP^{d-g}$'s.  Given the existence of the Jacobian $J^{d}_{X}$, the stronger  technical results about $A$  in \cite[Sect.~9.3]{kleiman05}, which were mentioned earlier, show that $A \colon X^{(d)} \to J^{d}_{X}$ is a $\bbP^{d-g}$-bundle.  

If we do not assume that $J^{d}_{X}$ exists, then we can use the above fact about Cartier divisors to construct the scheme.  It is enough to construct $J^{d}_{X}$ for $d>>0$, and for such a $d$, we can define $J^{d}_{X}$ to be the quotient of  $X^{(d)}$ by the relation of linear equivalence.  For this to be a valid definition, however, we must show that the quotient of $X^{(d)}$ exists as a $k$-scheme.  

In general, the quotient of a $k$-scheme by an equivalence relation may not exist as a $k$-scheme, but linear equivalence is a particularly nice equivalence relation: the equivalence classes are smooth and projective subschemes of $X^{(d)}$.  More formally, the relation is smooth and projective in the sense that the graph $R \subset X^{(d)} \times X^{(d)}$ of the equivalence relation has the property that the two projection morphisms $R \to X^{(d)}$ are smooth and projective.  The quotient of a quasi-projective $k$-scheme by such an equivalence relation always exists, as can be shown using an elementary argument (that realizes the quotient as a subscheme of a suitable Hilbert scheme).  We direct the reader to \cite[Lem.~4.9]{kleiman05} for details about the construction of the quotient by a smooth and projective relation and to \cite[Thm.~4.8]{kleiman05} for the proof that the quotient of $X^{(d)}$ is $J^{d}_{X}$.  In any case, this gives a second construction of the Jacobian.

The analytic construction from the beginning of this section showed that not only does the Jacobian exist, but it has the structure of a $g$-dimensional complex torus.  We can now show that $J^{0}_{X}$ is the algebro-geometric analogue of a $g$-dimensional complex torus: $J^{0}_{X}$ is a $g$-dimensional abelian variety.  Recall this means $J^{0}_{X}$ is a smooth proper variety of dimension $g$ that has group scheme structure.  The group structure on $J^{0}_{X}$ comes from the tensor product, and one can deduce the remaining properties from the fact that $X^{(d)}$ is a smooth projective variety of dimension $d$ by using $A \colon X^{(d)} \to J^{d}_{X}$ for $d >> 0$.

So far we have just made use of the Abel maps $A \colon X^{(d)} \to J^{d}_{X}$ for $d$ sufficiently large, but the Abel maps for small values of $d$ are also of interest.  The Abel map of degree $d=g$ is particularly interesting because $X^{(g)} \to J^{g}_{X}$ is birational. We can thus describe the Jacobian $J^{g}_{X}$ in terms of the symmetric power $X^{(g)}$ and the exceptional locus of the Abel map $A \colon X^{(g)} \to J^{g}_{X}$.

The Abel map of degree $d=g-1$ is also distinguished.  
\begin{definition}
	The image of $A \colon X^{(g-1)} \to J_{X}^{g-1}$ is the \textbf{theta divisor} $\Theta$.
\end{definition}
The map $A \colon X^{(g-1)} \to J^{g-1}_{X}$ is birational onto its image, so $\Theta \subset J^{g-1}_{X}$ is a subscheme of codimension $1$, hence is a locally principal divisor that can be shown to be ample.   Historically, the geometric study of Jacobian varieties has been dominated by the study of their theta divisors.  We conclude this section by showing  how the geometry of $\Theta$ encodes the geometry of $J_{X}^{d}$ by describing the theta divisor for curves of genus at most $4$.  

Somewhat more generally, we will describe the Abel maps
\begin{gather*}
	A \colon X^{(g)} \to J^{g}_{X} \text{ and }\\
	A \colon X^{(g-1)} \to \Theta.
\end{gather*}
The locus $C^{1}_{d} \subset X^{(d)}$ of points $[D]$ satisfying $\dim |D| \ge 1$ is closed, and the restriction of the Abel map to the complement
$X^{(d)} \setminus C^{1}_{d}$ is an isomorphism onto its image.  We will describe the structure of the Abel map by describing the subset $C^{1}_{d} \subset J^{d}_{X}$ and the contraction $C^{1}_{d} \to J^{d}_{X}$.

We derive such a description using the Riemann--Roch Formula.  By that formula, an effective divisor $D$ of degree $d=g-1$ satisfies $\dim |D| \ge 1$ (i.e.~$[D] \in C_{g-1}^{1}$) if and only if $\dim |\operatorname{adj} D| \ge 1$ or equivalently there are two distinct canonical divisors that contain $D$.  Similarly, an effective divisor $D$ of degree $d=g$ satisfies $\dim |D| \ge 1$ if and only if $|\operatorname{adj} D| \ge 0$ or equivalently some canonical divisor contains $D$.  

These observations demonstrate the key role played by the  canonical divisor $K$  in describing $C^{1}_{d}$.  The canonical divisor of a low genus curve is well-understood.  Indeed, effective canonical divisors are the preimages of hyperplanes under the canonical map, and the canonical map of a curve of genus at most $4$ is computed in  \cite[Chap.~4, Sect.~6]{hartshorne77}.  We now describe the Jacobian of a low genus curve.

\begin{example}[Genus $1$] \label{Example: SmGenus1}
The canonical divisor of a genus $1$ curve $X$ is trivial.  In particular, $C^{1}_{g}=\emptyset$, and so the Abel map  $A \colon \Hilb^{1}_{X}=X \to J^{1}_{X}$ is an isomorphism.  Now $J^{1}_{X}$ does not have natural group structure, but $J^{0}_{X}$ does, so if we fix a point $p_0 \in X$, then we can identify $X = J^{1}_{X}$ with $J^{0}_{X}$ by tensoring with $\calO_{X}(-p_0)$, and this identification makes $X$ into a group with identity $p_0$.  In a different form, this group law is often introduced in a first course in algebraic geometry.  The map associated to the complete linear system $| 3 p_0 |$ embeds $X$ as a cubic curve in the plane, and the group law coming from the isomorphism $X \cong J^{0}_{X}$ is the group law that is defined using the tangent-chord construction (as in, say, \cite[p.~321]{hartshorne77}).  

The theta divisor $\Theta \subset J^{0}_{X}$ is not very interesting.  The only effective degree $0$ divisor is the empty divisor, so we have $\Theta = \{ [\calO_{X} ] \}$.  
\end{example}

\begin{example}[Genus $2$]  \label{Example: SmGenus2}

Every genus $2$ curve is hyperelliptic, and the effective canonical divisors $K$ are the fibers of the degree $2$ map $f \colon X \to \mathbb{P}^1$ to the line.  These fibers $K=f^{-1}(t)$ are exactly the effective divisors of degree $g=2$ that move in a positive dimensional linear system.  Indeed, if $[D] \in C^{1}_{g}$, then, as we observed earlier, $D$ is contained in a canonical divisor, and hence equal to a canonical divisor by degree considerations.  This classification shows that $C^{1}_{g}$ is a rational curve
\begin{displaymath}
	\bbP^1 = C^{1}_{g} \subset X^{(g)},
\end{displaymath}
and this curve is contracted to a point by the Abel map.

What about $\Theta \subset J^{g-1}_{X}$?  No degree $1=g-1$ divisor $D$ can satisfy  $\dim |D|  \ge 1$.  Indeed, again using the observations we made after stating the Riemann--Roch Formula, we see that any such divisor would be contained in two distinct fibers of $f$, which is absurd.  We can conclude that the Abel map $A \colon X^{(g-1)} = X \to J^{g-1}_{X}$ is injective, so 
the theta divisor is 
\begin{displaymath}
	X = \Theta \subset J^{g-1}_{X},
\end{displaymath}
an embedded copy of the curve.
\end{example}

\begin{example}[Genus $3$] \label{Example: SmGenus3}
The genus $3$ curves fall into two classes: the hyperelliptic curves and the non-hyperelliptic curves.  We will first consider the case of hyperelliptic curves, which are similar to the genus $2$ curves that we just discussed.

Let $X$ be a genus $3$ hyperelliptic curve with degree $2$ map $f \colon X \to \mathbb{P}^{1}$ to the line.  The effective canonical divisors of $X$ are the divisors of the form $K = f^{-1}(t_1) + f^{-1}(t_2)$ for $t_1, t_2 \in \mathbb{P}^{1}$.  From this description, we can conclude that the elements $[D]$ of $C^{1}_{g}$ are the divisors of the form $f^{-1}(t_0) + p_0$ for $t_0 \in \bbP^1$ and $p_0 \in X$.  Furthermore, these divisors satisfy
\begin{displaymath}
	\dim | f^{-1}(t_0) + p_0| = 1
\end{displaymath}
by the Riemann--Roch Formula (as there is a unique canonical divisor containing $f^{-1}(t_0)+p_0$).  Indeed, the divisors $f^{-1}(t)+p_0$ with $t \in \mathbb{P}^{1}$ exhaust the effective divisors linearly equivalent to $f^{-1}(t_0) + p_0$.

How does this classification of divisors translate into a description of $A \colon X^{(g)} \to J^{g}_{X}$?  The exceptional locus $C^{1}_{g}$ is isomorphic to the surface $X \times \bbP^1$, and the Abel map collapses this surface to the curve $X$ by projection onto the first factor.

What about the theta divisor $\Theta \subset J^{g-1}_{X}$?  The exceptional locus $C^{1}_{g-1}$ of $A \colon X^{(2)} \to \Theta$ is the locus of effective degree $g-1=2$ divisors $D$ that are contained in two distinct canonical divisors.  From our description of $K$, we see that these are exactly  the divisors of the form $f^{-1}(t)$, $t \in \bbP^1$, and all these divisors are linearly equivalent.  In other words, the exceptional locus is a rational curve
\begin{displaymath}
	\bbP^1 = C^1_{g-1} \subset X^{(g-1)}
\end{displaymath}
that is contracted to a point.  Set-theoretically, this is the same as the description of the degree $g$ Abel map $X^{(g)} \to J^{g}_{X}$ of a genus $2$ curve.  There is, however, an important difference: when $X$ is a genus $2$ curve, the image of the rational curve $\mathbb{P}^{1} \subset X^{(2)}$ is  a smooth point, but when $X$ is a genus $3$ hyperelliptic curve, the image is a singularity of $\Theta$.  This is a consequence of a general result --- the Riemann Singularity Theorem --- that computes the multiplicity of $\Theta$ as
\begin{displaymath}
	\operatorname{mult}_{[L]} \Theta = h^{0}(X,L).
\end{displaymath}
Two proofs of this result can be found in \cite[Chap.~6]{ACGH}.

What about the non-hyperelliptic curves?  The canonical map of a non-hyperelliptic curve $X$ of genus $3$ embeds $X$ as a degree $4$ plane curve $X \subset \mathbb{P}^{2}$, and the effective canonical divisors are just the divisors that are the intersection of $X$ with a line $\ell \subset \mathbb{P}^{2}$.  From this description of the canonical divisors, we see that $[D] \in C^1_{g}$ if and only if $D$ lies on a line $\ell$ (which is necessarily unique).  We can define a map $\pi \colon C^1_{g} \to X$ as follows.  Given $[D] \in C^1_g$, there is a unique line $\ell$ containing $D$, and we can write $\ell \cap X = D + q_0$ for some point $q_0$ of $X$.  We set $\pi([D])=q_0$.

The fiber $\pi^{-1}(q_0)$ of $\pi$ over a point is a $\bbP^1$, the projective line parameterizing lines $\ell \subset \bbP^2$ containing $q_0$.  In particular, we see that $C^{1}_{g}$ is a surface.  The Abel map $A \colon X^{(g)} \to J^{g}_{X}$ contracts the fibers of $\pi$, so the image of $C^{1}_{g}$ is a curve:
\begin{displaymath}
	X = A(C^{1}_{g}) \subset J^{g}_{X}.
\end{displaymath}

We now turn our attention to the theta divisor.  We have $C^1_{g-1} = \emptyset$.  Indeed, no effective degree $g-1=2$ divisor is contained in two distinct canonical divisors because two distinct lines meet in a single point.  In particular, the theta divisor is a smooth projective surface:
\begin{displaymath}
	X^{(g-1)} = \Theta \subset J^{g-1}_{X}.
\end{displaymath} 
This is a special case of Marten's Theorem.  That theorem states that if $X$ is a curve of genus $g$, $g \ge 3$, then we have
\begin{displaymath}
	\dim \Theta_{\text{sing}} = \begin{cases} 		
							g-3 & \text{if $X$ is hyperelliptic;}\\
							g-4	& \text{otherwise.}
						\end{cases}
\end{displaymath} 
This statement, along with various generalizations, is proven in \cite[Chap.~4, Sect.~5]{ACGH}.
\end{example}

\begin{example}[Genus $4$] \label{Example: SmGenus4}
The structure of the Abel map of a genus $4$ hyperelliptic curve is similar to the structure of the Abel map of a genus $3$ hyperelliptic curve, so we will only discuss the non-hyperelliptic case.  The canonical map of a non-hyperelliptic genus $4$ curve $X$ embeds the curve in space $X \subset \mathbb{P}^{3}$.  As a space curve, $X$ is the complete intersection of a (non-unique) cubic hypersurface and a (unique) quadric surface, which we denote by $Q$.  The quadric surface $Q$ is either smooth or a cone over a plane curve of degree $2$.  The shape of the quadric influences the structure of the Abel map, so we consider these two cases separately.

Suppose first that $Q$ is smooth.  Then the quadric $Q$ must be isomorphic to $\bbP^{1} \times \bbP^{1}$ embedded by the complete linear system $|\calO(1,1)|$.  The description of $C^{1}_{g}$ is similar to the description we gave for a non-hyperelliptic genus $3$ curve.  By the Riemann--Roch Formula, the exceptional locus $C^{1}_{g} \subset X^{(g)}$ of $A \colon X^{(g)} \to J^{g}_{X}$ is the locus of effective degree $4$ divisors that are contained in a hyperplane  $h \subset \bbP^{3}$, which is necessarily unique (for otherwise $D$ would be contained in a line that lies on $Q$ by degree considerations, but the intersection of a line lying on $Q$ with $X$ is a divisor of degree $3$).  We describe $C^{1}_{g}$ by constructing a map $\pi \colon C^{1}_{g} \to X^{(2)}$.  Given $[D] \in C^{1}_{g}$, let $h \subset \bbP^3$ be the unique hyperplane containing $D$.  The intersection $h \cap X$ is a degree $6$ effective divisor that we can write as $h \cap X = D + E$ for some effective divisor $E$ of degree $2$. We set $\pi([D]) = [E]$.  The fibers of the resulting map $\pi \colon C^{1}_{g} \to X^{(2)}$ are $\bbP^1$'s, so $C^{1}_{g}$ is a smooth $3$-fold.  The Abel map contracts $C^{1}_{g}$ to the surface 
$$
	X^{(2)} =A(C^{1}_{g}) \subset J^{g}_{X}
$$
by contracting each fiber $\pi^{-1}([E])$ to a point.  

What about the degree $3=g-1$ Abel map $A \colon X^{(g-1)} \to J^{g-1}_{X}$?  If $[D]  \in X^{(g-1)}$ satisfies $\dim |D| \ge 1$, then $D$ lies on two distinct hyperplane sections, and hence lies on their intersection which is a line $\ell \subset \bbP^3$.  We can construct divisors satisfying this condition by using the geometry of the quadric surface $Q$.  Given a line $\ell$ on the quadric surface, the intersection $D := \ell \cap X$ is a degree $3$ effective divisor, so $[D] \in C^{1}_{g-1}$.  In fact, these points exhaust $C^{1}_{g-1}$.  If $[D] \in C^{1}_{g-1}$, then the unique line $\ell$ containing $D$ must be contained in $Q$ because otherwise $Q \cap \ell$ would be a degree $2$ closed subscheme that contains the degree $3$ effective divisor $D$, which is absurd.  The lines on $Q$ consist of two $1$-dimensional linear systems (the lines $\{t\} \times \bbP^{1}$ and the lines $\bbP^{1} \times \{ t \}$).  We thus have
\begin{displaymath}
	\bbP^1 \cup \bbP^1 = C^{1}_{g-1}\subset X^{(3)}.
\end{displaymath}
Each curve is contracted to a point of $\Theta$ which is a singularity.  This shows that $\Theta$ is singular at exactly two points.

What about the case where the quadric $Q$ containing $X$ is singular? The structure of the degree $g$ Abel map $A \colon X^{(g)} \to J^{g}_{X}$ is as before; the map contracts a threefold $C^{1}_{g} \subset X^{(g)}$ that is a $\bbP^1$-bundle over $X^{(2)}$.  The structure of the degree $g-1$ Abel map $A \colon X^{(g-1)} \to \Theta$, however, is different.  The argument used in the previous case remains valid except now $Q$ contains only one $1$-dimensional linear system of lines.  Recall that $Q$ is the cone over a smooth plane quadric $Y \subset \bbP^3$.  The lines $\ell$ on $Q$ are the lines that join a point on $Y$ to the vertex of the cone.  Thus
\begin{displaymath}
	\bbP^1 = C^{1}_{g} \subset X^{(g-1)}.
\end{displaymath} 
This curve is contracted to the unique singularity of $\Theta$.  (Warning: We have only defined $C^{1}_{g}$ as a set, and the locus naturally has non-reduced scheme structure.)  
\end{example}

This completes our discussion of Jacobians of smooth curves. The remainder of the article is devoted to extending this theory of the Jacobian to singular curves.

\section{Generalized Jacobians of Singular Curves} \label{Sec: GeneralizedJac}
How should the Jacobian of a singular curve  be defined?  One approach is to simply repeat Definition~\ref{Def: Jacobian}, which is the definition of the Jacobian of a smooth curve.  

\begin{definition} \label{Def: GenJacobian}
	Given a curve $X$, the \textbf{generalized Jacobian functor} $J_{X}^{d, \sharp}$ of degree $d$ is defined to be the \'{e}tale sheaf  associated to the functor that assigns to a $k$-scheme $T$ the set of isomorphism classes of lines bundles $L$ on $X_{T}$ that have the property the restriction of $L$ to any fiber of $X_{T} \to T$ has degree $d$.  The \textbf{generalized Jacobian variety} is the $k$-scheme that represents $J^{d, \sharp}_{X}$.
\end{definition}
This generalized Jacobian can be constructed by, for example, using Artin's Criteria.  However, the generalized Jacobian has a major deficiency: it is not proper.   

Consider the generalized Jacobian of a genus $1$ curve $X$ with a node $p_0 \in X$.  We observed in Example~\ref{Example: SmGenus1} that the degree $1$ line bundles on a smooth genus $1$ curve are all of the form $\calL(p)$ for $p \in X$ a point, and this fact remains valid for $X$ provided we require that $p$ lies in the smooth locus.  Thus we have 
\begin{align*}
	J_{X}^{1}=&	X \setminus \{ \text{node} \} \\
		\cong&	\bbP^1 \setminus \{ 0, \infty \}.
\end{align*}
In particular, $J_{X}^{1}$ is not proper.  This suggests a question: how to compactify $J_{X}^{1}$ to a proper $k$-scheme?

We answer this question in Section~\ref{Sect: CompJac}.  Before studying compactifications of $J_{X}^{d}$, let's first examine the structure of the scheme.  The generalized Jacobian $J^{0}_{X}$ is a smooth connected $g$-dimensional quasi-projective variety that admits a  group scheme  structure coming from the tensor product.  This group structure can be described in terms of the singularities of $X$.

When $k=\bbC$, we described the group of $k$-points $J^{0}_{X}(k)$ of the Jacobian cohomologically as a kernel in Eq.~\eqref{Eqn: LineBundlesAsCoh}. In that displayed equation, the terms in the exact sequence were cohomology groups that we interpreted as sheaf cohomology computed with respect to the analytic (or classical)  topology and the curve $X$ was a smooth curve.  However, our computation remains valid if we let $X$ be a singular curve over an arbitrary $k$ provided we  replace the analytic topology with the \'{e}tale topology.  Let us explain this.

In what follows, we only use the very basic properties of the \'{e}tale topology, and the reader unfamiliar with this formalism is directed to  \cite[Sect.~8.1]{bosch:1990}.  We write $T_{\text{\'{e}t}}$ for the \'{e}tale site of a scheme.  Given a smooth curve $X$ over $k=\bbC$, the isomorphism Eq.~\eqref{Eqn: LineBundlesAsCoh} is constructed by first constructing an isomorphism between the group of line bundles of arbitrary degree and the group $H^{1}(X_{\text{an}}, \calO_{X}^{\ast})$ and then observing that this isomorphism identifies the degree map on line bundles with the Chern class map $c_{1} \colon H^{1}(X_{\text{an}}, \calO_{X}^{\ast}) \to H^{2}(X_{\text{an}}, \bbZ)$.  These facts remain valid when $X$ has singularities and $k=\bbC$.  If  $k$ is an arbitrary field, then $X_{\text{an}}$ does not make sense, but there is an isomorphism between the group of line bundles and the \'{e}tale cohomology group $H^{1}(X_{\text{\'{e}t}}, \calO_{X}^{\ast})$ that identifies the degree map with a map valued in $H^{2}(X_{\text{\'{e}t}}, \bbZ_{\ell}(1))$ (for $\ell$ a prime distinct from the characteristic).  More generally, if we define $\operatorname{Pic}^{\sharp}_{X/k}$ to be the disjoint union of all the $J^{d,\sharp}_{X}$'s, then this functor is isomorphic to the 1st higher direct image of $\calO_{X}^{\ast}$ under the structure morphism $X_{\text{\'{e}t}} \to k_{\text{\'{e}t}}$.  This result is due to Grothendieck, and a recent discussion of the identification can be found in \cite[Sect.~2]{kleiman05}.  

The reason for introducing this cohomological formalism is that it makes it easy to relate the generalized Jacobian $J^{0}_{X}$ of $X$ to the Jacobian $J^{0}_{X^{\nu}}$ of the normalization $X^{\nu}$.  Let $\nu \colon X^{\nu} \to X$ be the normalization map.  The homomorphism $\nu^{-1} \calO_{X}^{\ast} \to \calO_{X^{\nu}}$ given by pulling back by $\nu$ is adjoint to a homomorphism $\calO_{X}^{\ast} \to \nu_{*}(\calO_{X^{\nu}}^{\ast})$ that is injective and an isomorphism away from the singular locus.  The cokernel $\calF$ is supported on $X_{\text{sing}}$ and fits into a short exact sequence 
\begin{equation} \label{Eqn: ShortComputePic}
	0 \to \calO_{X}^{\ast} \to \nu_{*}(\calO_{X^{\nu}}^{\ast}) \to \calF \to 0.
\end{equation}

Consider the associated long exact sequence relating the higher direct images of these sheaves under $X_{\text{\'{e}t}} \to k_{\text{\'{e}t}}$.  We just explained that the 1st direct image of $\calO_{X}^{\ast}$ is $\operatorname{Pic}^{\sharp}_{X/k}$.  Similarly, the 1st direct image of $\nu_{*} \calO_{X}^{\ast}$ is $\operatorname{Pic}^{\sharp}_{X^{\nu}/k}$ as $\nu$ is finite.  The natural map $\operatorname{Pic}^{\sharp}_{X/k} \to \operatorname{Pic}^{\sharp}_{X^{\nu}/k}$ is surjective as the cokernel injects into the 1st direct image of $\calF$, which is zero as $\calF$ has $0$-dimensional support.  For the same reason, the direct image of $\calF$ is the locally constant sheaf $F$ associated to the group $H^{0}(X_{\text{\'{e}t}}, \calF)$. The connecting homomorphism $F \to \operatorname{Pic}_{X/k}^{\sharp}$ is injective as $\calO_{X}^{\ast} \to \nu_{*} \calO_{X^{\nu}}^{\ast}$ is easily seen to induce an isomorphism on direct images (the only global functions on $X$, $X^{\nu}$ are constants).  To summarize, we can extract from the long exact sequence on higher direct images a short exact sequence
\begin{equation} 
	0 \to F  \to  \operatorname{Pic}^{\sharp}(X/k) \to \operatorname{Pic}^{\sharp}(X^{\nu}/k) \to 0.
\end{equation}
As $\operatorname{Pic}^{\sharp}(X/k) \to \operatorname{Pic}^{\sharp}(X^{\nu}/k)$ respects degree maps, we can also write
\begin{equation} \label{Eqn: LongComputePic}
	0 \to F  \to  J^{0}_{X} \to J^{0}_{X^{\nu}} \to 0.
\end{equation}
This is the desired description of the generalized Jacobian $J_{X}^{0}$: it is an extension of the abelian variety $J^{0}_{X^{\nu}}$ by the commutative group variety $F$  Let us examine the structure of $F$.

Label the singularities of $X$ as $q_1, \dots, q_n$ and then label the points of the fiber $\nu^{-1}(q_i)$ as $p_{i,j}$, $j=1, \dots, m_{i}$.  We defined $F$ to be the locally constant \'{e}tale sheaf associated to $H^{0}(X_{\text{\'{e}t}}, \calF)$, and this group of sections is isomorphic to the direct sum of the stalks of  $\calF$:
\begin{equation} \label{Eqn: StrOfKernel}
	H^{0}(X_{\text{\'{e}t}}, \calF) = \bigoplus (\calO_{X, p_{i,1} }^{\ast} \oplus \dots \oplus \calO_{X, p_{i,m_{i}}}^{\ast})/\calO_{X,q_i}^{\ast}.
\end{equation}
Here the stalks are taken with respect to the \'{e}tale topology, and the group $\calO^{\ast}_{X, q_i}$ is embedded diagonally.

The quotients appearing in Eq.~\eqref{Eqn: StrOfKernel} are unchanged if we replace the unit groups $\calO^{\ast}_{X, q_i}$ and $\calO^{\ast}_{X^{\nu}, p_{i,j}}$ with the unit groups
$\widehat{\calO}^{\ast}_{X, q_i}$ and $\widehat{\calO}^{\ast}_{X^{\nu}, p_{i,j}}$ in the appropriate completions.  In particular, the structure of $F$ depends only on the analytic type of the singularities of $X$.  Let's now compute $F$ for some specific singularities.

\begin{example}[The Node]
The node $\calO = k[[x,y]]/(xy)$ has normalization isomorphic to the product $\widetilde{\calO} = k[t]] \times k[[t]]$.  An isomorphism 
\begin{displaymath}
	\widetilde{\calO}^{\ast}/\calO^{\ast} \cong k^{\ast}
\end{displaymath}
is given by the rule $(f,g) \in \widetilde{\calO}^{\ast} \mapsto f(0)/g(0)$.  In other words, the associated algebraic group is the multiplicative group $\bbG_{m}$.  More generally, if $X$
is a nodal curve, then $F = \ker(J^{0}_{X} \to J^{0}_{X^{\nu}})$ is a multiplicative torus $\bbG_{m}^{\delta}$ of dimension equal to the number of nodes $\delta$.
\end{example}

\begin{example}[The Cusp]
If $\calO = k[[x,y]]/(y^2-x^3)$ is the cusp, then the quotient $\widetilde{\calO}^{\ast}/\calO^{\ast}$ is isomorphic to the additive group $k^{+}$.  In particular, if $X$ is a curve with only cusps as singularities, then $F$ is an additive torus $\bbG_{a}^{\delta}$ of dimension equal to the number of cusps $\delta$.
\end{example}

\begin{example}[The Tacnode]
Like the node, the normalization of the tacnode $\calO = k[[x,y]]/(y^2-x^4)$ is the product $\widetilde{\calO} = k[[t]] \times k[[t]]$ of two power series rings.  An isomorphism $\widetilde{\calO}^{\ast}/\calO^{\ast} \cong k^{\ast} \oplus k^{+}$ is given by the rule $(f,g) \mapsto (f(0)/g(0), f'(0)/f(0) - g'(0)/g(0))$.  Here $f'(t)$ and $g'(t)$ are the formal derivatives.  In particular, each tacnode of $X$ contributes a factor of $\bbG_{m} \times \bbG_{a}$ to $F$.
\end{example}

This concludes our discussion of examples.  A  more detailed description of $F$ can be found in  \cite[Sect.~9.1]{bosch:1990}.  
Rather than discussing that description, we turn our attention to the problem of compactifying the generalized Jacobian.

\section{Compactified Jacobians of Singular Curves} \label{Sect: CompJac}
In the special case of the genus $1$ nodal curve $X$ from the beginning Section~\ref{Sec: GeneralizedJac}, it is  clear how to compactify the associated generalized Jacobian $J_{X}^{1}$.
We described that generalized Jacobian as
\begin{equation} \label{Eqn: Genus1GenJac}
	J_{X}^{1} \cong X \setminus \{ \text{node} \},
\end{equation}
and we can construct a compactification $\bar{J}_{X}^{1} \supset J_{X}^{1}$ by defining
\begin{equation} \label{Eqn: Genus1CpctJac}
	\bar{J}_{X}^{1} \cong X.
\end{equation}
In this section, we will describe how to interpret this compactification $\bar{J}_{X}^{1}$ in terms of moduli and then how to generalize that interpretation to arbitrary curves.  

We can give a moduli-theoretic interpretation of the compactification \eqref{Eqn: Genus1CpctJac} as follows.  Let $p_0 \in X$ be the node.  In the isomorphism \eqref{Eqn: Genus1GenJac}, a point $p \in X \setminus \{ p_0 \}$ corresponds to $[\calL(p)] \in J^{1}_{X}$, where $\calL(p)$ is the $\calO_{X}$-linear dual of the ideal $I_{p}$ of $p$.  When $p=p_0$, the ideal $I_{p_0}$ is not a line bundle, but we can still define $\calL(p_0)$ by $\calL(p_0) := \ShHom(I_{p_0}, \calO_{X})$.  The compactification $\bar{J}_{X}^{1}$ parameterizes degree $1$ line bundles on $X$ and the sheaf $\calL(p_0)$.  Both the line bundles $\calL(p)$ and the sheaf $\calL(p_0)$ are examples of rank $1$, torsion-free sheaves, a class of sheaves that we now define.

\begin{definition}
	A coherent sheaf $I$ on a  curve $X$ is \textbf{torsion-free} if for every local  section $f$  of $\calO_{X}$ and every local section $s$ of $I$ satisfying
	$f \cdot s = 0$, we have $f=0$ or $s=0$.  The sheaf $I$ is said to be \textbf{rank $1$} if there exists a dense open subset $U \subset X$ such that $I|_{U}$ is isomorphic
	to $\calO_{U}$.
\end{definition}
The degree of a rank $1$, torsion-free sheaf is defined as follows.
\begin{definition}
	If $I$ is a rank $1$, torsion-free sheaf, then the \textbf{degree} $\deg(I)$ of $I$ is defined by
\begin{displaymath}
	\deg(I) := \chi(I)-\chi(\calO_{X}).
\end{displaymath}
\end{definition}

The compactified Jacobian is defined in the expected manner.
\begin{definition} \label{Def: CompactJacobian}
	The degree $d$ \textbf{compactified Jacobian functor} $\bar{J}_{X}^{d, \sharp}$ is defined to be the \'{e}tale sheaf associated to the functor that assigns to a $k$-scheme $T$ the set of isomorphism classes of $\calO_{T}$-flat, finitely presented $\calO_{X_T}$-modules $I$ on $X_{T}$ that have the property that the restriction of $I$ to any fiber of $X_{T} \to T$ is a rank $1$, torsion-free sheaf of degree $d$.  The degree $d$ \textbf{compactified Jacobian} is the $k$-scheme that represents $\bar{J}^{d, \sharp}_{X}$.
\end{definition}

By \cite[Thm.~8.1]{altman80}, the compactified Jacobian exists as a projective $k$-scheme.  Since every line bundle is a rank $1$, torsion-free sheaf, we have $J_{X}^{d} \subset \bar{J}^{d}_{X}$, and so $\bar{J}^{d}_{X}$ compactifies the generalized Jacobian $J^{d}_{X}$.  This compactification is particularly well behaved when $X$ has only planar singularities.  For such a curve, Altman--Iarrobino--Kleiman have proven that  the compactified Jacobian $\bar{J}_{X}^{d}$ is an irreducible variety that contains the generalized Jacobian as a dense open subset  \cite{Irred1977}.

The compactified Jacobian has undesirable properties when $X$ has a non-planar singularity.  In this case, Kleiman--Kleppe \cite{Kleppe81} have proven that $\bar{J}^{d}_{X}$ has at least two irreducible components.  The generalized Jacobian is contained in a single irreducible component, so the compactified Jacobian of a curve with a non-planar singularity has the undesirable property of having extra components.

The compactified Jacobian of a singular curve can be constructed using techniques similar to those used to construct the Jacobian of a smooth curve.  In Section~\ref{Sect: Jacobian}, we sketched two different proofs that the Jacobian of a non-singular curve exists. The first proof used Artin's Criteria to construct the Jacobian.  The author expects this proof can be modified to prove the existence of the compactified Jacobian of a singular curve (but additional care must be taken in showing that the relevant algebraic space is actually a scheme --- the compactified Jacobian is not a group scheme).

In any case, we are more interested in generalizing the second construction of the Jacobian.  In the second construction, given a smooth curve $X$, we fixed a sufficiently large integer $d$ and then constructed $J^{d}_{X}$ as the quotient of  the symmetric power $X^{(d)}$ by the relation of linear equivalence.  The quotient scheme exists essentially because the relevant equivalence classes are $\bbP^{d-g}$'s, and the  quotient map $X^{(d)} \to J^{d}_{X}$ is then the Abel map. 

In \cite{altman80},   Altmann--Kleiman construct the compactified Jacobian by extending the theory of linear equivalence and of  the Abel map to singular curves.  This extension is, however, non-trivial.  As we now demonstrate, the most naive approach to constructing the Abel map of a singular curve fails.

Given a singular curve $X$, let $X^{(d)}_{\text{sm}}$ denote the $d$-th symmetric power of the smooth locus of $X$.  Applied to $X^{(d)}_{\text{sm}}$, the construction from the proof of Lemma~\ref{Lemma: SymIsHilb} produces a regular map
\begin{displaymath}
	A \colon X^{(d)}_{\text{sm}} \to J^{d}_{X}
\end{displaymath}
or equivalently a rational map
\begin{displaymath}
	A \colon X^{(d)} \dashrightarrow \bar{J}^{d}_{X}.
\end{displaymath}
However, this second rational map may not be a regular map; the locus of indeterminacy may be non-empty.  We show this by example.

\begin{example} \label{Example: NaiveAbelBad}
	We exhibit a singular curve $X$ with the property that the map $A \colon X^{(2)} \dashrightarrow \bar{J}^{2}_{X}$ is not a regular map.  Let $X$ be a general genus $2$ curve that has a single node $p_0 \in X$.  That is, let $X$ be the projective curve that contains $U := \Spec(k[x,y]/(y^2 - x^2 (x-a_1) \dots (x-a_4) )$ (for some distinct, nonzero constants $a_1, \dots, a_4 \in k$) and is smooth away from $U$.  The maximal ideal $(x,y) \subset \calO_{U}$ corresponds to the node $p_0 \in X$.  The curve $X$ admits a degree $2$ morphism $\pi \colon X \to \bbP^1$ that extends the ring map $k[t] \rightarrow \calO_{U}$ defined by $t \mapsto x$.  We will show that $A \colon X^{(2)} \dashrightarrow \bar{J}^{2}_{X}$ is undefined at the point $[2 p_0] \in X^{(2)}$.

Our strategy is to construct two maps $f_1, f_2 \colon \Spec(k[[t]]) \to X^{(2)}$ out of the formal disc with the property that the closed point $0 \in \Spec(k[[t]])$ is sent to $[2 p_0]$ and the generic point does not map into the locus of indeterminacy of $A$.  Since $A$ is well-defined at the image of the generic point, the composition $A \circ f_i \colon \Spec(\operatorname{Frac} k[[t]]) \to \bar{J}^{2}_{X}$ is defined and hence extends to a morphism $\Spec(k[[t]]) \to \bar{J}^{2}_{X}$ by the valuative criteria of properness.  Write $\widetilde{g}_i$ for this extension.  We will show directly that $\widetilde{g}_{1}(0) \neq \widetilde{g}_{2}(0)$.  However, if $A$ was a regular map, then we would have $\widetilde{g}_{1}(0) = A( [2 p_0] )$ and $\widetilde{g}_{2}(0) = A( [2 p_0] )$, which is absurd.

We now  construct the maps $f_1, f_2 \colon \Spec( k[[t]] ) \to X^{(2)}$ by exhibiting corresponding algebra maps $(\calO_{U} \otimes \calO_{U})^{\operatorname{Sym}_{2}} \to k[[t]]$. Set $\sqrt{ (t-a_1) \dots (t-a_4)}$ equal to the power series that is the usual Taylor series expression for the square root of $(t-a_1) \dots (t-a_4)$.  (Choose $\sqrt{a_i}$ for $i=1, \dots, 4$ arbitrarily.)  

The map $f_1$ is defined to be the regular map that corresponds to the algebra homomorphism $(\calO_{U} \otimes \calO_{U})^{\operatorname{Sym}_{2}} \to k[[t]]$ that is the restriction of the map $\calO_{U} \otimes \calO_{U} \to k[[t]]$ defined by 
\begin{gather*}
	x \otimes 1\mapsto 0, \\
	y \otimes 1\mapsto 0, \\
	1 \otimes x \mapsto t, \\
	1 \otimes y \mapsto t \sqrt{ (t-a_1) \dots (t-a_4)}.
\end{gather*}
Intuitively, this is the formal arc that sends a parameter $t$ to an unordered pair of points that consists of $p_0$ and a point in $X$ that tends to $p_0$ as $t$ tends to $0$.

Similarly, we define $f_2$ to be the regular map that corresponds to the algebra homomorphism defined by
\begin{gather*}
	x \otimes 1 \mapsto t, \\
	y \otimes 1 \mapsto t \sqrt{ (t-a_1) \dots (t-a_4)}, \\
	1 \otimes x \mapsto t, \\
	1 \otimes y \mapsto -t \sqrt{ (t-a_1) \dots (t-a_4)}. \\
\end{gather*}
Intuitively, this is the formal arc that sends the parameter value $t$ to the two points in $\pi^{-1}(t)$.

Both maps have the property that $0$ maps to $[2 p_2]$ and  the generic point maps to a point where $X^{(d)} \dashrightarrow \bar{J}^{2}_{X}$ is regular.  To complete this example, we need to show that $\widetilde{g}_{1}(0) \ne \widetilde{g}_{2}(0)$.  

By construction, the composition $\Spec( \operatorname{Frac} k[[t]] ) \stackrel{f_1}{\longrightarrow} X^{(2)} \stackrel{A}{\dashrightarrow}  \bar{J}^{2}_{X}$ corresponds to the line bundle on $X \otimes \operatorname{Frac} k[[t]]$ that is the pullback of $\calO(1)$ under $\pi \otimes 1 \colon X \otimes \operatorname{Frac} k[[t]] \to \bbP^{1} \otimes \operatorname{Frac} k[[t]]$.  This line bundle extends to the pullback of $\calO(1)$ under $\pi \otimes 1 \colon X \otimes k[[t]] \to \bbP^1 \otimes k[[t]]$, and so the extension $\widetilde{g}_1$ of $A \circ f_1$ must satisfy $\widetilde{g}_1(0)=[\pi^{*}(\calO(1))]$.  The composition $A \circ f_2 \colon \Spec( \operatorname{Frac} k[[t]]) \to \bar{J}^{2}_{X}$ corresponds to a sheaf that fails to be locally free.  This property persists under specialization, so if $\widetilde{g}_{2}(0) = [I]$, then $I$ must fail to be locally free.  In particular, $\widetilde{g}_2(0) \ne \widetilde{g}_{1}(0)$.  We can conclude that $A \colon X^{(2)} \to \bar{J}^{2}_{X}$ is not a regular map.
\end{example}

The above example shows that, to define a suitable Abel map for a singular curve, we must  replace $X^{(d)}$ with a blow-up that resolves the indeterminacy of $X^{(d)} \dashrightarrow \bar{J}_{X}^{d}$.  In Section~\ref{Section: Abel Map}, we resolve this indeterminacy by a blow-up that is a moduli space that parameterizes generalized divisors, objects we discuss in the next section.

\section{Generalized Divisors} \label{Section: GeneralizedDivisors}
Here we develop the theory of generalized divisors in analogy with the theory of Cartier divisors.  Recall that on a smooth curve  a line bundle is equivalent to a linear equivalence class of Cartier divisors (see Sect.~\ref{Sect: Jacobian}).  Our goal is to define more general divisors on a singular curve so that a rank $1$, torsion-free sheaf is equivalent to a linear equivalence class of these more general divisors.  We define two types of divisors that generalize Cartier divisors: generalized divisors and generalized $\omega$-divisors.  On a Gorenstein curve, generalized divisors are essentially equivalent to generalized $\omega$-divisors, but the two divisors are fundamentally different on a non-Gorenstein curve, and it is only generalized $\omega$-divisors that behave as one expects by analogy with Cartier divisors.  The material in this section is derived from Hartshorne's papers \cite{hartshorne86}, \cite{hartshorne94}, and \cite{hartshorne07}.

\subsection{Generalized Divisors} \label{Subsect: GenDivisor}
Let $X$ be an integral curve.  We begin by defining a generalized divisor on $X$.
\begin{definition}
	  Let $\calK$ denote the sheaf of total quotient rings of $X$.  That is, $\calK$ is the field of rational functions, considered as a locally constant sheaf.  A \textbf{generalized divisor} $D$ on $X$ is a nonzero subsheaf $I_{D} \subset \calK$ that is a coherent $\calO_{X}$-module. If additionally $I_{D}$ is a line bundle, then we say that $D$ is a \textbf{Cartier divisor}.   We say that a generalized divisor $D$ is \textbf{effective} if $I_{D} \subset \calO_{X}$.
 \end{definition}
An effective generalized divisor on $X$ is just a $0$-dimensional closed subscheme $Z \subset X$.  In particular, every closed point $p \in X$ defines a generalized divisor that, by abuse of notation, we denote by $p$.  A second source of generalized divisors is rational functions.  A rational function $f$ generates a subsheaf $\calO_{X} \cdot f \subset \calK$, and we write $\operatorname{div}(f)$ for the corresponding generalized divisor.  The generalized divisor $\operatorname{div}(1)$ is written $0$.

Many of the familiar operations on Cartier divisors extend to generalized divisors.
\begin{definition} \label{Def: SumAndMinus}
	The \textbf{sum} $D+E$ of two generalized divisors $D$ and $E$ is the subsheaf $I_{D+E} \subset \calK$ generated by local sections  of the form $f g$
with $f \in I_{D}$ and $g \in I_{E}$.  The \textbf{minus} $-D$ of  a generalized divisor $D$ is the subsheaf $I_{-D} \subset \calK$ whose local sections are elements $f \in \calK$ satisfying $f \cdot I_{D} \subset \calO_{X}$.

\end{definition}
\begin{remark}
	Our definitions of sum and minus are different from the definitions found on  \cite[p.~88]{hartshorne07}.  There Hartshorne defines $I_{D+E}$ to be the $\omega$-reflexive hull of the module generated by the elements $f g$ rather than the module itself and similarly with $I_{-D}$.  His definition is, however, equivalent to the one just given because the module generated by the $f g$'s is $\omega$-reflexive by Lemma~\ref{Lemma: StarReflexive} below.
\end{remark}

The sum operation is easily seen to make the set of generalized divisors into a commutative monoid with identity $0$.  
More precisely, properties of sum and minus are summarized by the lemma below. 
\begin{lemma} \label{Lemma: DivisorSum}
	The sum and minus operations have the following properties:
\begin{enumerate}
	\renewcommand{\theenumi}{(\alph{enumi})}
	\renewcommand{\labelenumi}{\theenumi}
	
	\item sum is associative and commutative;
	\item $D+0=D$;
	\item $D + -D = 0$ provided $D$ is Cartier; \label{Lemma: DivisorSumOne}  
	\item $-(D+E) = -D + -E$ provided $E$ is Cartier. \label{Lemma: DivisorSumTwo}  
\end{enumerate}
\end{lemma}
\begin{proof}
	This is \cite[Prop.~2.2]{hartshorne07}.
\end{proof}
In the last two parts of Lemma~\ref{Lemma: DivisorSum}, it is necessary to assume that one of the divisors is Cartier.  The sum operation is not well-behaved when 
applied to non-Cartier divisors, as is shown by the examples in  \cite[Sect.~3]{hartshorne94}.

We now define linear equivalence.
\begin{definition}
	Two generalized divisors $D$ and $E$ are \textbf{linearly equivalent} if there exists a rational function $f$ with $D=E+\operatorname{div}(f)$.  Given $D$, the associated \textbf{complete linear system} $|D|$ is the set of all effective generalized divisors linearly equivalent to $D$.  
\end{definition}
Linear equivalence is immediately seen to define an equivalence relation on the set of generalized divisors.   The lemma below provides an alternative characterization of linear equivalence.

\begin{lemma} \label{Lemma: PropertiesGeneralizedDivisor}
	The following relations between generalized divisors and rank $1$, torsion-free sheaves hold:
\begin{enumerate}
	\renewcommand{\theenumi}{(\alph{enumi})}
	\renewcommand{\labelenumi}{\theenumi}
	
	\item $D$ is linearly equivalent to $E$ if and only if $I_{D}$ is isomorphic to $I_{E}$; \label{Item: DivisorPropertyIso}
	\item every rank $1$, torsion-free sheaf $I$ is isomorphic to $I_{D}$ for some generalized divsior $D$; \label{Item: DivisorPropertyExist}
	\item $I_{D} \otimes I_{E} \cong I_{D+E}$ provided $E$ is Cartier; \label{Item: DivisorPropertySum} 
	\item $I_{-D} \cong \ShHom(I_{D},\calO_{X})$. \label{Item: DivisorPropertyDiff}
\end{enumerate}
\end{lemma}
\begin{proof}
	Conditions~\ref{Item: DivisorPropertyIso} and \ref{Item: DivisorPropertyExist} are \cite[Prop.~2.4]{hartshorne07}, and the reader may find a  proof of Condition~\ref{Item: DivisorPropertyDiff}  in \cite[Lem.~2.2]{hartshorne94}.   To prove Condition~\ref{Item: DivisorPropertySum}, it is enough to prove that the natural surjection  $I_{D} \otimes I_{E} \to I_{D+E}$ is injective.  Injectivity can be checked locally, so we may assume $I_{E}$ is principal, in which case injectivity is immediate.  
\end{proof}

To study $|D|$, we make the following definition.
\begin{definition}
	Write $I^{\vee} := \ShHom(I, \calO_{X})$ for the \textbf{$\calO_{X}$-linear dual} of a coherent $\calO_{X}$-module $I$.  The sheaf $\calL(D)$ associated to a generalized divisor $D$ is defined by $\calL(D) := I_{D}^{\vee}$. An \textbf{adjoint generalized divisor} $\operatorname{adj} D$ of the generalized divisor $D$ is a generalized divisor satisfying $\calL(\operatorname{adj} D )=\ShHom(\calL(D),\omega)$.  A \textbf{canonical generalized divisor} $K$ is an adjoint divisor of $0$ (so $\calL(K)=\omega$).
\end{definition}

A canonical divisor exists when $\omega$ is reflexive (e.g.~when $X$ is Gorenstein).  Indeed, by Lemma~\ref{Lemma: PropertiesGeneralizedDivisor} we can write $I_{K} = \omega^{\vee}$ for some generalized divisor $K$.  We then have $\calL(K)= (\omega^{\vee})^{\vee}$, which equals $\omega$ by reflexivity.  Thus $K$ is a canonical divisor.  

When $X$ is Gorenstein,  $\omega$ is not only reflexive but in fact locally free, so $K$ exists and is Cartier.  The canonical divisor $K$ is base-point free when $g \ge 1$ by \cite[Thm.~1.6]{hartshorne86}, and so it determines a \textbf{canonical map} $X \to \bbP H^{0}(X, \calL(K))^{\vee} = \bbP^{g-1}$.  The image of this morphism is a curve provided $g \ge 2$, and in this case, we define the image to be the \textbf{canonical curve}.  There is a definition of the canonical curve and the canonical map of a non-Gorenstein curve (see e.g.~\cite{kleiman09}), but the definitions are more complicated, and we do not study them here.

The assumption that $X$ is Gorenstein also implies that the adjoint divisor of a given divisor $D$ exists as we can take $\operatorname{adj} D := K+ -D$.  Just as with a smooth curve, the elements of $|\operatorname{adj} D|$ 
are the preimages of hyperplanes in canonical space $\bbP^{g-1}$ that contain $D$.

When $X$ is non-Gorenstein, an adjoint divisor $\operatorname{adj} D$ of a divisor $D$ may not exist, and when it exists, it may not be unique even up to linear equivalence because we cannot  recover $I_{\operatorname{adj} D}$ from its dual $\calL(\operatorname{adj} D)$.  We can, however, recover $I_{\operatorname{adj} D}$ when this sheaf is reflexive because then $I_{\operatorname{adj} D} \cong \calL(\operatorname{adj} D)^{\vee}$.  The following lemma shows that $I_{\operatorname{adj} D}$ is always reflexive when $X$ is Gorenstein.

\begin{lemma} \label{Lemma: GorensteinMeansReflexive}
	If $X$ is Gorensten, then every rank $1$, torsion-free sheaf is reflexive.  That is, the natural map $I \to (I^{\vee})^{\vee}$ is an isomorphism.  In fact, we have
	\begin{displaymath}
	D = -(-D)
	\end{displaymath}
for all generalized divisors $D$.
\end{lemma}
\begin{proof}
	This is \cite[Lem.~1.1]{hartshorne86}. 
\end{proof}

The lemma is false if we omit the hypothesis that $X$ is Gorenstein.
\begin{example} \label{Example: NonReflexiveDualizing}
	We provide an example of a curve whose dualizing sheaf is not reflexive.  Take $X$ to be the curve in Example~\ref{Example: Dualizing}.  This is a genus $2$ curve with a unique singularity $p_0$ that is unibranched and non-Gorenstein.

We compute $\omega^{\vee}$ and $(\omega^{\vee})^{\vee}$.  Let $\calK_{\omega}$ denote the locally constant sheaf of rational $1$-forms.  Write $\partial_{t}$ for the functional $(\calK_{\omega})|_{X_1} \to \calK|_{X_1}$ that sends $dt$ to $1$ and similarly for $\partial_{s} \colon (\calK_{\omega})|_{X_2} \to \calK|_{X_2}$.  A computation shows that the image of any functional $\omega|_{X_1} \to \calO_{X_1}$ is contained in the subsheaf of regular functions that vanish at the singularity $p_0$, so $\omega^{\vee}$ can be described as the sheaf generated by $t^{6} \partial_{t}, t^{7} \partial_{t}, t^{8} \partial_{t}$  on $X_1$ and by $\partial_{s}$ on $X_2$.  The same reasoning shows that $(\omega^{\vee})^{\vee}$ is generated by $dt/t^{3}, dt/t^{2}, dt/t$ on $X_1$ and by $ds$ on $X_2$.  In particular, $\omega \to (\omega^{\vee})^{\vee}$ is not an isomorphism because $dt/t$ is not in the image. 
\end{example}

The curve from the above example does not admit a canonical divisor.  Indeed, we have seen that the reflexivity of $\omega$ is a sufficient condition for the existence of $K$, but it is also a necessary condition because every dual module is reflexive. As the dualizing module of the curve in Example~\ref{Example: NonReflexiveDualizing} is not reflexive, this curve does not admit a canonical divisor.

We now develop the properties of complete linear systems.  The sheaf $\calL(D) = \ShHom(I_{D}, \calO_{X})$ has the property that the natural map 
\begin{displaymath}
	H^{0}(X, \ShHom(I_{D}, \calO_{X})) \to \Hom(I_{D}, \calO_{X})
\end{displaymath}
is an isomorphism.  (Indeed, this holds quite generally when $I_{D}$ is replaced by an arbitrary coherent sheaf.)  We use this observation in the following lemma to describe $|D|$ in terms of a cohomology group.
\begin{lemma} \label{Lemma: DivisorToSectionGor}
	Let $D$ be a generalized divisor.  Then the rule that sends a nonzero global section of $\calL(D)$ to the image of the corresponding homomorphism $I_{D} \to \calO_{X}$  defines a bijection
	\begin{equation} \label{Eqn: DivisorToSectionGor}
		\bbP H^{0}(X, \calL(D)) \cong |D|.
	\end{equation}
\end{lemma}
\begin{proof}
	Under the hypothesis that $X$ is Gorenstein, this is stated on \cite[p.~378, top]{hartshorne86}, but the fact remains valid when $X$ is non-Gorenstein. To show that the map is well-defined, we need to show that any nonzero homomorphism $\phi \colon I_{D} \to \calO_{X}$ is injective.  Such a $\phi$ must be generically nonzero because $\calO_{X}$ is torsion-free.  Because both $I_{D}$ and $\calO_{X}$ are rank $1$, we can conclude  that $\phi$ is in fact  generically an isomorphism.  In particular, the kernel is supported on a proper closed subset of $X$.  But the only such subsheaf of $I_{D}$ is the zero subsheaf as $I_{D}$ is torsion-free, so $\phi$ is injective.  
	
	Given that the map \eqref{Eqn: DivisorToSectionGor} is well-defined, it is immediate that the map is surjective.  To show injectivity, we argue as follows.  Suppose that we are given two nonzero global sections $\sigma_1$ and $\sigma_2$ that correspond to  two homomorphisms $\phi_1, \phi_2 \colon I_{D} \to \calO_{X}$ with the same image.  We have just shown that $\phi_{2}$ is injective, so this homomorphism is an isomorphism onto its image.  If we write $\phi_{2}^{-1} \colon \operatorname{im}(\phi_1) = \operatorname{im}(\phi_2) \to I_{D}$ for the inverse homomorphism, then $\alpha := \phi^{-1}_{2} \circ \phi_{1}$ satisfies $\phi_{1} = \phi_{2} \circ \alpha$.  The only automorphisms of $I_{D}$ are the maps given by multiplication by nonzero scalars  by \cite[Cor.~5.3, Lem.~5.4]{altman80}.  If $\alpha$ is given by multiplication by $c \in k^{\ast}$, then  $\phi_{1} = c \cdot \phi_{2}$ or equivalently $\sigma_1 = c \cdot \sigma_2$.  In other words, $\sigma_1$ and $\sigma_2$ define the same point of $\bbP H^{0}(X, \calL(D))$, showing injectivity.  This completes the proof.
\end{proof}

Motivated by the previous lemma, we now study the cohomology of $\calL(D)$.  The most important numerical invariant controlling the cohomology is the degree.

\begin{definition}
	The \textbf{degree} $\deg(D)$ of a generalized divisor $D$ is defined by $\deg(D) := -\deg(I_{D})$.  
\end{definition}

\begin{lemma} \label{Lemma: DegreeOfDivisor}
The degree function has the following properties:
	\begin{enumerate}
		\renewcommand{\theenumi}{(\alph{enumi})}
		\renewcommand{\labelenumi}{\theenumi}
		
		\item if $D$ is an effective generalized divisor, then $\deg(D)$ is the length of $\calO_{D}$; \label{Lemma: DegreOfDivisorOne}
		\item $\deg(D+E) = \deg(D) + \deg(E)$ provided $E$ is Cartier; \label{Lemma: DegreOfDivisorTwo} 
		\item $\deg(-D) = -\deg(D)$ provided $X$ is Gorenstein; \label{Lemma: DegreOfDivisorThree}
		\item if $D$ is linearly equivalent to $E$, then $\deg(D)=\deg(E)$. \label{Lemma: DegreOfDivisorFour}
	\end{enumerate}
\end{lemma}
\begin{proof}
Property~\ref{Lemma: DegreOfDivisorOne} follows from the additivity of the Euler characteristic $\chi$, and Property~\ref{Lemma: DegreOfDivisorFour} is immediate from Lemma~\ref{Lemma: PropertiesGeneralizedDivisor}\ref{Item: DivisorPropertyIso}.  Property~\ref{Lemma: DegreOfDivisorTwo} follows from the more general identity $\chi(I_{D} \otimes I_{E}^{\otimes n}) = n \deg(I_{E}) + \chi(I_{D})$ which is e.g.~p.~295 and Corollary 2, p.~298 of \cite[Chap.~1]{kleiman66}.  Property~\ref{Lemma: DegreOfDivisorTwo} and coherent duality imply Property~\ref{Lemma: DegreOfDivisorThree}.  Indeed, if $X$ is Gorenstein, then $\ShHom(I_{D}, \omega) = I_{D}^{\vee} \otimes \omega$.  By \ref{Lemma: DegreOfDivisorTwo}, $ \deg(\omega) + \chi(I_{D}^{\vee}) = \chi(I_{D}^{\vee} \otimes \omega)$, and $\chi(\ShHom(I_{D}, \omega)) = -\chi(I_{D})$ by coherent duality.  We now deduce \ref{Lemma: DegreOfDivisorThree} by elementary algebra.

\end{proof}
\begin{remark}
	When $X$ is Gorenstein, Hartshorne defines the degree differently on \cite[p.~2]{hartshorne86}, but the two definitions coincide by  \cite[Prop.~2.16]{hartshorne94}.
\end{remark}

One consequence of Lemma~\ref{Lemma: DegreeOfDivisor}\ref{Lemma: DegreOfDivisorThree} is that 
$$
	\deg(\calL(D))=\deg(D)
$$
 when $X$ is Gorenstein (as $I_{-D} = \calL(D)$).   When $X$ is non-Gorenstein, this equality can fail, as the example below shows.

\begin{example} \label{Example: WrongDegree}
	We give an example where $\deg(\calL(D)) \ne \deg(D)$ or equivalently $\deg(-D) \ne -\deg(D)$.  Consider the non-Gorenstein genus $2$ curve $X$ from Example~\ref{Example: Dualizing} and the generalized divisor $D=p_0$ that is the singularity.  We have $\deg(p_0)=1$, but we claim $\deg(-p_0)=-2$.  The sheaf $I_{p_0}$ is the ideal generated by $(t^{3}, t^{4}, t^{5})$ on $X_1$ and by $(1)$ on $X_2$.  A direct computation shows that $I_{-p_0}$ is the subsheaf of $\calK$ generated by $(1,t,t^{2})$ on $X_1$ and by $(1)$ on $X$.  The degree of $\calL(p_0)$ is $2= 1+\deg(p_0)$, not $\deg(p_0)$.  (To compute the degree, observe that $\calO_{X} \subset I_{-D}$ has colength $2$.)
\end{example}

While we always have $|D| = \bbP H^{0}(X, \calL(D))$, the dimension of $|D|$ is only well-behaved when $X$ is Gorenstein.  First, we have the following form of the Riemann--Roch Formula.
\begin{proposition} \label{Prop: GorensteinRiemanRoch}
	If $X$ is Gorenstein and $D$ is a generalized divisor, then we have
	\begin{displaymath}
		\dim |D| - \dim |\operatorname{adj} D| = \deg(D)+1-g.
	\end{displaymath}
\end{proposition}
\begin{proof}
	This is  \cite[Thm.~1.3, 1.4]{hartshorne86}, and with our  definition of degree, the equation is a consequence of Eq.~\eqref{SerreDuality}.
\end{proof}
We now use the Riemann--Roch Formula to compute $\dim |D|$.

\begin{corollary} \label{Cor: GorensteinGenLinSystem}
	Assume $X$ is Gorenstein.  Then  the equation 
\begin{equation} \label{Eqn: GeneralLinearSystemGorenstein}
	\dim |D|  = 	\begin{cases}
					-1 & \text{if $d <g$;} \\
					d-g& \text{if $d \ge g$.}
				\end{cases}
\end{equation}
	 holds for every generalized divisor $D$ of degree $d > 2g-2$.  
	 
	 Furthermore, if $D$ is a divisor of degree  $d \le 2g-2$, then there exists a degree $0$ Cartier divisor $E$ such that $D+E$ satisfies Eq.~\eqref{Eqn: GeneralLinearSystemGorenstein}.
\end{corollary}
\begin{remark}
	In Section~\ref{Section: Abel Map}, we will introduce the moduli space of effective generalized divisors of degree $d$, called the Hilbert scheme $\operatorname{Hilb}^{d}_{X}$.  The present corollary implies that the locus of divisors satisfying  Eq.~\eqref{Eqn: GeneralLinearSystemGorenstein} is dense in $\Hilb^{d}_{X}$ (i.e.~Eq.~\eqref{Eqn: GeneralLinearSystemGorenstein} is satisfied by a general effective $D$ of degree $d$). 
\end{remark}
\begin{proof}

Suppose first that $D$ is a generalized divisor of degree $d>2g-2$.  Then the degree of $\operatorname{adj} D$ is $2g-2-d<0$, and it follows from degree considerations that the only homomorphism $I_{\operatorname{adj} D} \to \calO_{X}$ is the zero homomorphism.  In other words, $|\operatorname{adj} D|=\emptyset$, and the claim immediately follows  from the Riemann--Roch Formula.  

Now suppose that $D$ is a given generalized divisor of degree $d \le 2g-2$.  Choose an integer $e$ large enough so that $d+e > 2g-2$ and then choose a Cartier divisor $A$ that is the sum of $e$ distinct points lying in the smooth locus $X^{\text{sm}}$.  We claim that there exist distinct points $p_1, \dots, p_e \in X^{\text{sm}}$ such that $D+A-p_1-\dots-p_e$ satisfies Eq.~\eqref{Eqn: GeneralLinearSystemGorenstein}.  

We construct the $p_i$'s by induction.  It is enough to show that if $B$ is a generalized divisor with $\dim |B|>1$, then there exists a point $p \in X^{\text{sm}}$ with  $h^{0}(X, \calL(B-p))=h^{0}(X, \calL(B))-1$.  Fix a nonzero section $\sigma \in H^{0}(X, \calL(B))$.  Then the image $\sigma(p) \in k(p) \otimes \calL(B)$ of $\sigma$ in the fiber $\calL(B)$ at $p$ is zero for only finitely many points $p$.  Pick a point $p$ in the smooth locus of $X$ such that $\sigma(p) \ne 0$.  Then $h^{0}(X, \calL(B-p))=h^{0}(X, \calL(B))-1$ because $H^{0}(X, \calL(F-p))$ is the kernel of the nonzero homomorphism $H^{0}(X, \calL(F)) \to k(p) \otimes \calL(F)$.  This completes the proof.
\end{proof}

The corollary is false without the Gorenstein assumption. 

\begin{example} \label{Example: LargeNonspecial}
	Let $X$ be the non-Gorenstein genus $2$ curve from Example~\ref{Example: Dualizing}.  If $p_0 \in X$ is the singularity, then for any $4$ general points $p_1, p_2, q_1, q_2 \in X$, the divisors $p_0+p_1+p_2$ and $p_0+q_1+q_2$ are linearly equivalent. Indeed, assume the points all lie in $X_1$ and are respectively given by the ideals $(t-a_1), (t-a_2), (t-b_1), (t-b_2)$.  Then 
\begin{displaymath}
	p_0+p_1+p_2 = p_0+q_1+q_2 + \operatorname{div}(f) \text{ for $f = \frac{(t-a_1)(t-a_2)}{(t-b_1)(t-b_2)}$.}
\end{displaymath}
This shows that $\dim |p_0+p_1+p_2| \ge 2 >1=d-g$ for $d = \deg(p_0 + p_1 + p_2)$ even though $d>2g-2$.  
\end{example}

A more satisfactory theory of linear systems on a non-Gorenstein curve can be constructed by changing the definition of divisor.  

\subsection{Generalized $\omega$-divisors}

We now define generalized $\omega$-divisors and develop their properties in analogy with generalized divisors.  The definition is as follows.

\begin{definition}
	Let $\omega_{\calK}$ be the sheaf of rational $1$-forms.  A \textbf{generalized $\omega$-divisor} $D_{\omega}$ on $X$ is a nonzero subsheaf $I_{D_{\omega}} \subset \omega_{\calK}$ that is a coherent $\calO_{X}$-module.  We say that a generalized $\omega$-divisor $D_{\omega}$ is \textbf{Cartier} if $I_{D_{\omega}}$ is a line bundle.  A generalized $\omega$-divisor is \textbf{effective} if $I_{D_{\omega}} \subset \omega$.  
\end{definition}
(We denote generalized $\omega$-divisors with a subscript $\omega$ to avoid confusing them with generalized divisors.)

When $X$ is Gorenstein, the definition of a generalized $\omega$-divisor is essentially equivalent to the definition of a generalized divisor.  Indeed, for such an $X$, tensoring with the dualizing sheaf $\omega$ defines a bijection  between the set of generalized $\omega$-divisors and generalized divisors that preserves properties such as effectiveness.  No such bijection exists when $X$ is non-Gorenstein, and the remainder of this section is devoted to demonstrating to the reader that generalized $\omega$-divisors behave better than generalized divisors.

Unlike effective generalized divisors, effective generalized $\omega$-divisors do not correspond to closed subschemes of $X$.  However, given a point $p \in X$, the subsheaf $I_{p} \cdot \omega \subset \omega$ of $1$-forms that vanish at $p$ defines an effective $\omega$-divisor that we denote by $p_{\omega}$.  (Later in this section we will define the degree of an $\omega$-divisor, and the reader is warned that $p_{\omega}$ may not have degree $1$ when $p$ is a singularity.)  The submodule $\omega \subset \calK_{\omega}$ defines an effective generalized $\omega$-divisor that we denote $0_{\omega}$.  The $\omega$-divisor $\operatorname{div}(\eta)$ associated to a rational $1$-form $\eta \in \calK_{\omega}$ is defined by setting $\operatorname{div}(\eta) := \calO_{X} \cdot \eta \subset \calK_{\omega}$.

We can define basic operations on generalized $\omega$-divisors in analogy with the operations we defined on generalized divisors, although there are a few complications. 
\begin{definition}
	We define the \textbf{sum} $D+E_{\omega}$ of a generalized divisor $D$ and a generalized $\omega$-divisor $E_{\omega}$ by setting $I_{D+E_{\omega}} \subset \calK_{\omega}$ equal to the subsheaf generated by elements of the form $f \eta$ with $f$ a local section of $I_{D}$ and $\eta$ a local section of $I_{E_{\omega}}$.  (The sum of two generalized $\omega$-divisors is not well-defined.)  The \textbf{negation} $n(D_{\omega})$ of a generalized $\omega$-divisor is the generalized divisor defined by setting $I_{n(D_{\omega})} \subset \calK$ equal to the subsheaf generated by elements $f$ that have the property that $f \cdot I_{D_{\omega}} \subset \omega$.  Swapping the roles of $\calK$ and $\calK_{\omega}$, we get the definition of the negation $n(D)$ of a generalized divisor.
\end{definition}

\begin{lemma}
	The sum and negation operations have the following properties:
\begin{enumerate}
	\renewcommand{\theenumi}{(\alph{enumi})}
	\renewcommand{\labelenumi}{\theenumi}
	
	\item sum is associative (i.e.~$(D+E)+F_{\omega} = D+(E+F_{\omega})$;
	\item $0+D_{\omega}=D_{\omega}$;
 	\item $D+n(D)=0_{\omega}$ for every Cartier divisor $D$ and $n(D_{\omega})+D_{\omega}=0_{\omega}$ for every Cartier $\omega$-divisor $D_{\omega}$;
	\item $n(D+E_{\omega}) = n(E_{\omega}) + n(D)$ provided $D$ or $E_{\omega}$ is Cartier.  
\end{enumerate}
\end{lemma}
\begin{proof}
		The proof is analogous to the proof of Lemma~\ref{Lemma: DivisorSum}.  We leave the details to the interested reader.
\end{proof}

\begin{definition}
	Two generalized $\omega$-divisors $D_{\omega}$ and $E_{\omega}$ are said to be \textbf{linearly equivalent} if there exists a rational function $f$ such that $D_{\omega} = \operatorname{div}(f) + E_{\omega}$.  The set of all effective generalized $\omega$-divisors linearly equivalent to a given generalized $\omega$-divisor is written $|D_{\omega}|$ and called the associated \textbf{complete linear system}.
\end{definition}

\begin{lemma}
The following relations between generalized $\omega$-divisors and rank $1$, torsion-free sheaves hold:
\begin{enumerate}
	\renewcommand{\theenumi}{(\alph{enumi})}
	\renewcommand{\labelenumi}{\theenumi}

	\item two generalized $\omega$-divisors $D_{\omega}$ and $E_{\omega}$ are linearly equivalent if and only if $I_{D_{\omega}} \cong I_{E_{\omega}}$;
	\item every rank $1$, torsion-free sheaf is isomorphic to $I_{D_{\omega}}$ for some generalized $\omega$-divisor $D_{\omega}$;
	\item $I_{D+E_{\omega}} \cong I_{D} \otimes I_{E_{\omega}}$ provided either $D$ or $E_{\omega}$ is Cartier;
	\item $I_{n(D)} = \ShHom(I_{D}, \omega)$ and $I_{n(E_{\omega})} = \ShHom(I_{E_{\omega}}, \omega)$.
\end{enumerate}	
\end{lemma}
\begin{proof}
	The proof is entirely analogous to the proof of Lemma~\ref{Lemma: PropertiesGeneralizedDivisor}.
\end{proof}
Proceeding as we did with generalized divisors, we now define the canonical and adjoint divisors.

\begin{definition}
	Write $I^{\ast} := \ShHom(I, \omega)$ for the \textbf{$\omega$-linear dual} of a coherent $\calO_{X}$-module $I$.  The sheaf $\calM(D_{\omega})$ associated to a generalized $\omega$-divisor is defined by $\calM(D_{\omega})  := I_{D_{\omega}}^{\ast}$.  An \textbf{adjoint generalized $\omega$-divisor} $\operatorname{adj} D_{\omega}$ of a generalized $\omega$-divisor $D_{\omega}$ is a generalized $\omega$-divisor satisfying $\calM(\operatorname{adj} D_{\omega}) = \ShHom(\calM(D_{\omega}), \omega)$.  A \textbf{canonical generalized $\omega$-divisor} $K_{\omega}$ is an adjoint $\omega$-divisor of $0_{\omega}$ (so $\calM(K_{\omega}) = \omega$).
\end{definition}

The lemma below implies that an adjoint $\omega$-divisor of a generalized $\omega$-divisor exists and is unique up to linear equivalence.  In particular, every curve admits a canonical $\omega$-divisor $K_{\omega}$.  We do not study the relation between $K_{\omega}$ and the canonical map here.

\begin{lemma} \label{Lemma: StarReflexive}
	A rank $1$, torsion-free sheaf $I$ is $\omega$-reflexive.  That is, the natural map
\begin{displaymath}
	I \to (I^{\ast})^{\ast}
\end{displaymath}
is an isomorphism.  More generally, $n(n(D_{\omega}))=D_{\omega}$ for every generalized $\omega$-divisor $D_{\omega}$ and $n(n(D))=D$ for every generalized divisor $D$.
\end{lemma}
\begin{proof}
	This is  \cite[Lem.~1.4]{hartshorne07}.  
\end{proof}

We now develop the theory of the complete linear system associated to a $\omega$-divisor.
\begin{lemma}
	Let $D_{\omega}$ be  a generalized $\omega$-divisor.  Then the rule that  sends a nonzero global section of $\calM(D_{\omega})$ to the image of the corresponding homomorphism $I_{D_{\omega}} \to \omega$  defines a bijection
	$$
		\bbP H^{0}(X, \calM(D_{\omega})) \cong |D_{\omega}|.
	$$ 
\end{lemma}
\begin{proof}
	The proof is analogous to the proof of Lemma~\ref{Lemma: DivisorToSectionGor}.  The details are left to the interested reader.
\end{proof}

\begin{definition}
	The \textbf{degree} of a generalized $\omega$-divisor $D_{\omega}$ is defined by $\deg(D_{\omega}) := \deg(\omega)-\deg(I_{D_{\omega}})$.
\end{definition}

\begin{lemma} \label{Lemma: DegreeOfOmegaDivisor}
	The degree function has the following properties:
	\begin{enumerate}
		\renewcommand{\theenumi}{(\alph{enumi})}
		\renewcommand{\labelenumi}{\theenumi}

		\item if $D_{\omega}$ is an effective generalized $\omega$-divisor, then $\deg(D_{\omega})$ equals the length of the quotient module $\omega/I_{D_{\omega}}$; \label{Lemma: DegreeOfOmegaDivisorOne}
		\item $\deg(D+E_{\omega})=\deg(D)+\deg(E_{\omega})$ provided $D$ or $E_{\omega}$ is Cartier; \label{Lemma: DegreeOfOmegaDivisorTwo}
		\item $\deg( n(D) ) = -\deg(D)$ and $\deg( n(D_{\omega}) ) = -\deg(D_{\omega})$; \label{Lemma: DegreeOfOmegaDivisorThree}
		\item if $D_{\omega}$ is linearly equivalent to $E_{\omega}$, then $\deg(D_{\omega})=\deg(E_{\omega})$. \label{Lemma: DegreeOfOmegaDivisorFour}
	\end{enumerate}
\end{lemma}
	\begin{proof}
	The proof of this lemma is similar to the proof of Lemma~\ref{Lemma: DegreeOfDivisor}. 
\end{proof}
We can use the lemma to conclude that 
$$
	\deg(\calM(D_{\omega}))=\deg(D_{\omega}).
$$
For generalized divisors, the analogous equality only holds when $X$ is Gorenstein.  

We now state  the Riemann--Roch Formula.
\begin{proposition}
	We have
	\begin{displaymath}
		\dim |D_{\omega}| - \dim |\operatorname{adj} D_{\omega}| = \deg(D_{\omega})+1-g.
	\end{displaymath}
\end{proposition}
\begin{proof}
This is a consequence of coherent duality \eqref{OmegaSerreDuality}. 

\end{proof}

Having established a version of the Riemann--Roch Formula for generalized $\omega$-divisors, we can now describe $\dim |D_{\omega}|$ just as we did for divisors on a smooth curve.

\begin{corollary} \label{Cor: GenLinSystem}
The equation
\begin{equation} \label{Eqn: GeneralLinearSystemNonGorenstein}
	\dim |D_{\omega}|  = 	\begin{cases}
					-1 & \text{if $d <g$;} \\
					d-g& \text{if $d \ge g$.}
				\end{cases}
\end{equation}
holds for every generalized $\omega$-divisor $D_{\omega}$ of degree $d > 2g-2$

 Furthermore, if $D_{\omega}$ is a $\omega$-divisor of degree  $d \le 2g-2$, then there exists a degree $0$ Cartier divisor $E$ such that $E+D_{\omega}$ satisfies Eq.~\eqref{Eqn: GeneralLinearSystemNonGorenstein}. 
\end{corollary}
\begin{remark}
As with Corollary~\ref{Cor: GorensteinGenLinSystem}, the second part of the present corollary implies that the general effective  $\omega$-divisor satisfies Eq.~\eqref{Eqn: GeneralLinearSystemNonGorenstein}.  To be precise,  in Section~\ref{Section: Abel Map} we will introduce the Quot scheme $\Quot^{d}_{\omega}$ which is a projective $k$-scheme that parameterizes  effective generalized $\omega$-divisors of degree $d$, and  the present corollary implies that the locus of $\omega$-divisors satisfying  Eq.~\eqref{Eqn: GeneralLinearSystemNonGorenstein} is dense in $\Quot^{d}_{X}$.
\end{remark}

\begin{proof}
	Replace generalized divisors with $\omega$-divisors in the proof of Corollar~\ref{Cor: GorensteinGenLinSystem}.
\end{proof}

\subsection{Examples} \label{Subsect: Examples}
We now study  generalized divisors on a singular curve of low genus.  Our goal is to provide examples similar to those given at the end of Section~\ref{Sect: Jacobian} in order to illustrate the differences  and similarities between divisors on non-singular curves and divisors on singular curves. We focus on describing divisors of degree $d=g-1$ and $d=g$.  

\begin{example}[Genus $1$] \label{Example: SingGenus1}
A genus $1$ curve $X$ is always Gorenstein with trivial canonical divisor $K=0$. The only effective generalized divisor of degree $0=g-1$ divisor is the empty divisor $0$, and a degree $1=g$ effective divisor is a point $p \in X$.  No divisor of degree $g-1$ or $g$ moves in a positive dimensional linear system.
\end{example}

\begin{example}[Genus $2$] \label{Example: SingGenus2}
Here we first encounter non-Gorenstein curves.  Let us first dispense with the Gorenstein case, which is analogous to the smooth case.  If $X$ is a Gorenstein curve of genus $2$, then $X$ admits a degree $2$ morphism $f \colon X \to \bbP^1$ whose fibers $f^{-1}(t)$ are exactly the effective canonical divisors $K$ (\cite[Prop.~3.2]{kleiman09}).  Arguing as in Example~\ref{Example: SmGenus2}, we see that no degree $1=g-1$ divisor moves in a positive dimensional linear system, and the effective degree $2=g$ divisors that satisfy $\dim |D| \ge 1$ are exactly the effective canonical divisors. 

What about the non-Gorenstein curves?  We will just consider the  curve $X$ from Example~\ref{Example: Dualizing} (which is a representative example). Recall that $X$ is a rational curve with a unique  singularity $p_0 \in X$ that is unibranched and non-Gorenstein.  We saw in Example~\ref{Example: LargeNonspecial} that generalized divisors on $X$ do not behave as they do on a Gorenstein curve: there exists a degree $d=2g-1$ generalized divisor $D$ with $\dim |D| > d-g$, and this cannot happen on a Gorenstein curve.  Let us examine generalized divisors of degree $1=g-1$ and $2=g$.

A degree $1=g-1$ divisor $D=p$ must satisfy $\dim |D|=0$.  This is \cite[Thm.~8.8]{altman80}.  If $\dim |p|>0$, then $p$ would necessarily be linearly equivalent to a point $q$ lying in the smooth locus and any non-constant function with at worst a pole at $q$ would define an isomorphism $X \cong \bbP^1$, which is impossible.

There are degree $2=g$ divisors that move in a positive dimensional linear system.  If $p, q \in X$ are points distinct from the singularity $p_0$, then a modification of the construction from Example~\ref{Example: LargeNonspecial} shows that $p_0 + p$ is linearly equivalent to $p_0 +q$, so $\dim |p_0+p| \ge 1$.  In fact, $\dim |p_0+p|=1$.  We can prove this as follows.  Take $p_1$ to be the point in $X_2$ defined by the ideal $(s)$.  Then $I_{p_0+p_1}$ is the ideal generated by $(t^3, t^4, t^5)$ on $X_1$ and by $(s)$ on $X_2$.  A computation shows that $\calL(p_0+p_1)$ is the subsheaf of $\calK$ generated by $(1,t,t^2)$ on $X_1$ and by $(s^{-1})$ on $X_2$, and this rank $1$, torsion-free sheaf is isomorphic to the direct image $\nu_* \calO(1)$ of the line bundle $\calO(1)$ under the normalization map $\nu \colon \bbP^{1} = X^{\nu} \to X$.  In particular, $h^{0}(X, \calL(p_0+p_1)) = 2$, so $\dim |p_0+p_1| = 1$.  

What are all the elements of $|p_0+p_1|$?  In addition to the divisors $p_0+p$, the linear system $|p_0+p_1|$ contains a non-reduced divisor: the non-reduced divisor $D_0$ whose ideal $I_{D_0}$ is
$$
	I_{D_{0}} = \begin{cases}
				(t^{4}, t^{5}, t^{6})	& \text{ on $X_1$;} \\
				(1)				& \text{ on $X_2$}.
			\end{cases}
$$
Indeed, $D_0$ lies in $|p_{0}+p_1|$ as $D_0=p_0+p_1+\operatorname{div}(t^{-1})$.  The generalized divisors that we have constructed are all the effective degree $2$ divisors that move in a positive dimensional linear system.  Let us prove this statement.

There are no degree $2$ Cartier divisors that move in a positive dimensional linear system because the existence of such a divisor would imply the existence of a non-constant degree $2$ morphism $X \to \bbP^1$, forcing $X$ to be Gorenstein (\cite[Prop.~2.6]{kleiman09}). To handle non-Cartier divisors, we need to do more work.  

We claim that if $E$ is an effective generalized divisor of degree $2$ that does not lie in $|p_0+p_1|$ and is not Cartier, then $\calL(E)$ is the subsheaf of $\calK$ generated by $1,t,t^{2}$ on $X_1$ and by $1$ on $X_2$.  This subsheaf is isomorphic to the direct image $\nu_* \calO$ of the structure sheaf of the normalization, so any such $E$ satisfies $h^{0}(\calL(E))=1$ or equivalently $\dim |E|=0$.  Thus the claim implies that the elements of $| p_{0} + p_1 |$ are exactly the degree $2$ divisors that move in a positive dimensional linear system as we wished to show.

We prove the claim by direct computation.  If $E$ does not lie in $|p_0+p_1|$ and is not Cartier, then $E$ must be supported at the singularity $p_0$, so we can pass from $X$ to the open affine $X_1$.  By considering the dimension of $\calO_{E} = \calO_{X_1}/I_{E}$, we see that the ideal $I_{E}$ of $E$ must contain the square of the maximal ideal $(t^{3}, t^{4}, t^{5})$ and a $2$-dimensional subspace $W$ of $k \cdot t^{3} + k \cdot t^{4} + k \cdot t^{5}$.  Certainly the elements $1, t, t^{2}$ are all contained in $\calL(E)$, so $\calL(E)$ contains $\calO_{X_1}^{\nu}=k[t]$.  To complete the proof of the claim, we need to show that $\calL(E)$ is no larger.  

Thus suppose $f \in \calL(E)$. Write $f$ as a Laurent series in $t$.  Because $D_0 \ne E$, the vector space $W$ must contain an element of the form $t^{3}+a t^{4} + b t^{5}$.  If multiplication by $f \in k(X)$ maps $t^{3} + a t^{4} + b t^{5}$ into $\calO_{X_1}$, then the Laurent series of $f$ must be of the form  $c_{-3} t^{-3} + c_{0} + c_{1} t^{1} + \dots$.  However, $W$ must also contain an element of the form $c t^{4} + d t^{5}$ with $c$ and $d$ not both zero.  By examining the coefficients of $(c t^{4} + d t^{5}) f$, we see that $c_{-3}=0$.  This proves the claim, completing our discussion of generalized divisors on $X$.

Generalized $\omega$-divisors on $X$ are easier to analyze because we can use the Riemann--Roch Formula.  Effective canonical $\omega$-divisors $K_{\omega}$ are  $\omega$-divisors of degree $2=g$ that move in a positive dimensional linear system.  Indeed, $K_{\omega}$ has degree $2$ and satisfies
\begin{align*}
	\dim |K_{\omega}| =& \dim H^{0}(X, \calM(K_{\omega})) - 1\\
			=& \dim H^{0}(X, \omega) - 1\\
			=& 1.
\end{align*}
The canonical $\omega$-divisors are the only $\omega$-divisors of degree $g$ that move in a positive dimensional linear system.  If $\deg(D_{\omega})=g$ and $\dim | D_{\omega} |>0$, then  by the Riemann--Roch Formula, $\operatorname{adj} D_{\omega}$ is a  $\omega$-divisor of degree $0$ that satisfies $|\operatorname{adj} D_{\omega}| \ne \emptyset$.  This is only possible if $\operatorname{adj} D_{\omega}=0$ (as degree considerations show that any nonzero global section of $I_{\operatorname{adj} D_{\omega}}$ defines an isomorphism with $\calO_{X}$). We can conclude that $D_{\omega}=K_{\omega}$.

What about the generalized $\omega$-divisors of degree $1=g-1$?  A point $p \in X$ distinct from $p_0$ defines an effective $\omega$-divisor $p_{\omega}$ of degree $1$.  There are additional effective $\omega$-divisors of degree $1$ that correspond to quotients supported at the singularity $p_0$.  In Example~\ref{Example: Dualizing} of the Appendix we compute a presentation of $\omega$, and from that presentation, we see that the stalk $\omega/I_{p_0} \cdot \omega$ of $\omega$ at $p_0$ is a $2$-dimensional $k$-vector space. There is thus a $1$-dimensional family of quotients of $\omega$ supported at $p_0$, corresponding to the surjections $\omega/I_{p} \cdot \omega \to k$.

None of these $\omega$-divisors moves in a positive dimensional linear system.  Indeed, fix a point $q \in X$ distinct from the singularity $p_0$.  If $D_{\omega}$ is a\ $\omega$-divisor of degree $1=g-1$ with $\dim |D_{\omega}| > 0 $, then $-q+D_{\omega}$ is a degree $0$ $\omega$-divisor that is effective, so $-q+D_{\omega}$ is linearly equivalent to $0_{\omega}$.  We can conclude that $D_{\omega}$ is linearly equivalent to $q_{\omega}$.  In particular, $D_{\omega}$ must be Cartier.  As in the case of generalized divisors, if we fix two linearly independent global sections of $\calM(D_{\omega})$, then these sections define an isomorphism $X \cong \bbP^1$, which is impossible. This completes our study of genus $2$ curves.
\end{example}

\begin{example}[Genus $3$] \label{Example: SingGenus3}
The classification of  non-Gorenstein curves of genus $g$ becomes complicated once $g \ge 3$, so we now focus on the Gorenstein case.  As in the smooth case, if $X$ is a Gorenstein curve of genus $3$, then the canonical map $X \to \bbP^2$ is either an embedding of $X$ as a plane curve of degree $4$ or a degree $2$ map onto a plane quadric curve, in which case $X$ is hyperelliptic (by \cite[Thm.~4.13]{kleiman09}).  We begin by analyzing the hyperelliptic case.

If $X$ is hyperelliptic with degree $2$ map to the line $f \colon X \to \bbP^1$, then the effective canonical divisors are the divisors of the form $K=f^{-1}(t_1)+f^{-1}(t_2)$ for $t_1, t_2 \in \bbP^1$.  Arguing as in Example~\ref{Example: SmGenus3}, we see that the degree $2=g-1$ effective divisors that are contained in a positive dimensional linear system are the divisors $f^{-1}(t)$ with $t \in \bbP^1$, and the degree $3=g$ effective divisors with this property are the divisors $p+f^{-1}(t)$ for $p \in X$.  In particular, every degree $g-1$ effective generalized divisor that moves in a positive dimensional linear system is Cartier, but this is not true for degree $g$ divisors.  If $p_0 \in X$ is a singularity, then $p_0+f^{-1}(t)$ for $t \in \bbP^1$ is not Cartier but $\dim |p_0 + f^{-1}(t)|=1$.

What about when $X$ is non-hyperelliptic?  The curve is then a degree $4$ plane curve $X \subset \bbP^2$, and the effective canonical divisors are restrictions of lines $K = \ell \cap X$.  As in the hyperelliptic case, the analysis we gave for non-singular curves at the end of Section~\ref{Sect: Jacobian} extends to the present case.  No effective generalized divisor of degree $g-1$ moves in a positive dimensional linear system, but an effective generalized divisor of degree $g$ moves in a positive dimensional linear system when it is contained in a line, so e.g.~three points  $p_0, p_1, p_2$ of $X$ that lie on a line satisfy $\dim |p_0+p_1+p_2| =1$.  In particular, if $p_0$ is a node of $X$, then $p_0 + p_1 + p_2$ is an effective generalized divisor of degree $g$ that is not Cartier but moves in a positive dimensional linear system.

\end{example}

\begin{example}[Genus $4$] \label{Example: SingGenus4}
Here we only study a few specific examples of genus $4$ curves.  On a  curve of genus $g<4$, every effective generalized divisor $D$ of degree $g-1$ that moves in a positive dimensional linear system is Cartier.  This remains true on the general curve of genus $g=4$, but not on certain special genus $4$ curves.  We focus on describing divisors on these  special  curves.

A singular hyperelliptic curve of genus $4$ is an example of  a special curve.  If $f \colon X \to \bbP^1$ is the degree $2$ map to $\bbP^1$, then the effective canonical divisors are the divisors of the form $K=f^{-1}(t_1) + f^{-1}(t_2) + f^{-1}(t_3)$.  We can conclude that the degree $g-1$ divisors that move in a positive dimensional linear system are the divisors of the form $p_0+f^{-1}(t)$ for $p_0 \in X$ and $t \in \bbP^1$.  In particular, the divisor $p_0 + f^{-1}(t)$ is not Cartier when $p_0 \in X$ is a singularity.

What about the non-hyperelliptic case?  If $X$ is a non-hyperelliptic Gorenstein curve of genus $4$, then the canonical map realizes $X$ as a degree $6$ space curve $X \subset \bbP^3$.  Just as in the smooth case, $X$ is the complete intersection of a (non-unique) cubic hypersurface  and a (unique) quadric hypersurface $Q$.  If the two hypersurfaces are general, then $X$ is smooth, but $X$ can have singularities when the hypersurfaces are special.  Consider the case where $Q$ is the cone over a plane quadric curve $Y$ with  vertex $p_0 \in \bbP^3$ and the cubic hypersurface contains $p_0$ but is otherwise general. The curve $X$ then has a unique node that is the vertex $p_0$ of $Q$.

The effective canonical divisors of $X$ are the divisors of the form $K= h \cap X$ for $h \subset \bbP^3$ a hyperplane.  Our analysis in Example~\ref{Example: SmGenus4} shows that an effective degree $3=g-1$ generalized divisor $D$ moves in a positive dimensional linear system precisely when $D$ is contained in a line $\ell \subset Q$ that lies on the quadric surface.  Every such line meets $X$ in $3$ points, one of which is the vertex $p_0$.  The lines on $Q$ are exactly the lines joining a point of the plane curve $Y$ to the vertex $p_0$.  If the (set-theoretic) intersection of a line and $X$ consists of the points $p_0, p_1, p_2$, then $p_0+p_1+p_2$ is a degree $g-1$ effective generalized divisor that moves in a positive dimensional linear system, and the general  divisor with this property is of the form $p_0+p_1+p_2$ for a suitable line.  Because $p_0 \in X$ is a node, $p_0+p_1+p_2$ is not Cartier.  Thus every degree $g-1$ generalized divisor that moves in a positive dimensional linear system fails to be Cartier.
\end{example}

\section{The Abel map} \label{Section: Abel Map}
At the end of Section~\ref{Sect: CompJac}, we posed the problem of defining an Abel map associated to a singular curve.  There are two motivations.  First, a  theory of the Abel map provides us with a tool for constructing and describing the compactified Jacobian $\bar{J}_{X}^{d}$.  Second, a  theory of the Abel map allows us to define and study the theta divisor $\Theta \subset \bar{J}^{g-1}_{X}$ associated to a singular curve.   

Example~\ref{Example: NaiveAbelBad} demonstrated that the most naive approach to constructing an Abel map fails:  If $X^{(d)}$ is the $d$-th symmetric power of a curve, then the  rule that assign to $d$ general points $p_1, \dots, p_d$ the line bundle $\calO_{X}(p_1+\dots+p_d)$ defines a rational map
\begin{displaymath}
	A \colon X^{(d)} \dashrightarrow \bar{J}^{d}_{X},
\end{displaymath}
but this map may not be regular.  In the example, the 2nd Abel map $A \colon X^{(2)} \dashrightarrow \bar{J}^{2}_{X}$ is undefined at a unique point.  

Lemma~\ref{Lemma: SymIsHilb} suggests an alternative approach to constructing the Abel map.  For a smooth curve $X$, the symmetric power $X^{(d)}$ can be interpreted as the moduli space of effective Cartier divisors of degree $d$, and the Abel map $A \colon X^{(d)} \to J^{d}_{X}$ can be interpreted as the map that sends a divisor $D$ to its associated line bundle $\calL(D)$.  When $X$ is singular, the symmetric power no longer has such a moduli theoretic interpretation. Here we resolve the indeterminacy of $A \colon X^{(d)} \dashrightarrow \bar{J}^{d}_{X}$ by constructing an Abel map out of a moduli space that maps to $X^{(d)}$.

Our discussion from Section~\ref{Section: GeneralizedDivisors} suggests two possibilities for the moduli space: the moduli space of effective generalized divisors and the moduli space of effective generalized $\omega$-divisors.  We will see that there are two different Abel maps corresponding to the two different types of divisors.  The two maps are essentially equivalent when $X$ is Gorenstein but have  different properties in general.  We begin by defining the relevant moduli spaces.  

Effective generalized divisors are parameterized by the \textbf{Hilbert scheme} $\Hilb_{X}^{d}$ which is defined to be the $k$-scheme that represents the Hilbert functor $\Hilb_{X}^{d, \sharp}$ defined in Definition~\ref{Def: HilbFunctor}. (There  $X$ was assumed to be smooth, but the definition remains valid when $X$ is singular.)  Generalized $\omega$-divisors are parameterized by a  Quot scheme that we now define.
\begin{definition}
	The \textbf{Quot functor} $\Quot_{\omega}^{d, \sharp}$ of degree $d$ is defined by setting $\Quot_{\omega}^{d, \sharp}(T)$ equal to the set of isomorphism classes of $T$-flat quotients $q \colon \omega_{T} \twoheadrightarrow Q$ of the dualizing sheaf with the property that the restriction $Q$ to any fiber of $X_{T} \to T$ has length $d$.  The  \textbf{Quot scheme} $\Quot_{\omega}^{d}$ of degree $d$ is the $k$-scheme that represents $\Quot_{\omega}^{d, \sharp}$.
\end{definition}
A quotient map $q \colon \omega \to Q$ is equivalent to the inclusion $\ker(q) \hookrightarrow \omega$ which defines a generalized $\omega$-divisor, so $\Quot_{\omega}^{d}$ is the moduli space of effective generalized $\omega$-divisors of degree $d$.  For the remainder of the section, we will write $D_{\omega}$ rather than $[q]$ for an element of $\Quot_{\omega}^{d}$ to emphasize the connection with generalized $\omega$-divisors.

The moduli spaces $\Hilb^{d}_{X}$ and $\Quot^{d}_{\omega}$ both exist as projective $k$-schemes by a theorem of Grothendieck (\cite[Thm.~2.6]{altman80}).  The construction realizes both of these moduli spaces as suitable subschemes of a Grassmannian variety.  

The Hilbert scheme and the Quot scheme are related to the symmetric power by the Hilbert--Chow morphism.  Suppose that $D \subset X$ is a degree $d$ closed subscheme.  Given $p \in X$, write $\ell_{p}(D)$ for the length of $\calO_{D, p}$.  With this definition, we can associate to $D$ the point in the symmetric power 
\begin{equation} \label{Eqn: HilbChowRule}
	[\sum \ell_{p}(D) \cdot p] \in X^{(d)}.
\end{equation}
We would like to assert that there is a morphism 
\begin{displaymath}
	\rho \colon \Hilb^{d}_{X} \to X^{(d)}
\end{displaymath}
with the property that $\rho$ is defined by Eq.~\eqref{Eqn: HilbChowRule} as a set-map.  We also like to assert the existence of an analogous morphism $\rho \colon \Quot^{d}_{\omega} \to X^{(d)}$ out of the Quot scheme.

These assertions are surprisingly difficult to prove.  The difficulty is that our construction of $X^{(d)}$ does not provide direct access to the scheme's functor of points.  One approach to constructing $\rho$ is to describe the functor of points of $X^{(d)}$ in a way that makes it possible to interpret Eq.~\eqref{Eqn: HilbChowRule} as the definition of a transformation of functors.  A detailed discussion of this approach can be found in the thesis of David Rydh (\cite{rydh08a} and esp.~\cite[Sect.~6]{rydh08b}).

Rydh attributes this construction of $\rho$, which he calls the Grothendieck--Deligne norm map, to Grothendieck and Deligne (\cite[Expos\'{e}~XVII, 6.3.4]{SGA4}).  In this approach, one first identifies $X^{(d)}$ with the space of divided powers.  The space of divided powers has a moduli-theoretic interpretation in terms of multiplicative polynomial laws, and $\rho$ is then constructed as the morphism that sends a closed subscheme to the multiplicative law given by a determinant construction.
A similar construction produces a morphism $\rho \colon \Quot^{d}_{\omega} \to X^{(d)}$ out of the Quot scheme.  There are other approaches to constructing $\rho$, using e.g.~projective geometry, and we direct the interested reader to \cite[Sect.~6]{rydh08b} for a review of the relevant literature.

We will try to resolve the indeterminacy of the Abel map $A \colon X^{(d)} \dashrightarrow J^{d}_{X}$ by constructing morphisms $\Hilb^{d}_{X} \to \bar{J}^{d}_{X}$ and $\Quot^{d}_{\omega} \to \bar{J}^{d}_{X}$ that factor as $\Hilb^{d}_{X} \stackrel{\rho}{\longrightarrow} X^{(d)} \stackrel{A}{\dashrightarrow} \bar{J}^{d}_{X}$ and $\Quot^{d}_{X} \stackrel{\rho}{\longrightarrow} X^{(d)} \stackrel{A}{\dashrightarrow} \bar{J}^{d}_{X}$.  We will always be able to construct a suitable morphism out of $\Quot^{d}_{\omega}$, but we will only be able to construct a morphism out of $\Hilb^{d}_{X}$ when $X$ is Gorenstein.

Before constructing the morphisms, we warn the reader that $\rho$ is not always a resolution of indeterminacy in the usual sense because $\rho$ is not always birational.  The Hilbert--Chow morphism is an isomorphism over the locus parameterizing $d$ distinct points (\cite[Thm.~II.3.4]{iversen70}).  When the singularities of $X$ are planar, the locus of distinct points in $\Hilb^{d}_{X}$ is dense, so $\rho$ is a birational map.  The Hilbert--Chow morphism is not birational when $X$ has a singularity of embedding dimension $3$ or more.  Indeed, when $X$ has such a singularity, $\Hilb^{d}_{X}$ is reducible, and the locus of $d$ distinct points is contained in a single component.  The other components are collapsed by $\rho$, so $\rho$ is not birational.  The issue with $\Quot^{d}_{\omega}$ is similar.

In any case, we now define Abel maps associated to singular curves.
\begin{definition} \label{Def: AbelForSing}
	The \textbf{Abel map} $A_{h} \colon \Hilb^{d}_{X} \to \bar{J}^{-d}_{X}$ out of the Hilbert scheme is defined by the rule 
	$$
		[D] \mapsto [I_{D}],
	$$
where $I_{D}$ is the ideal of the closed subscheme $D \subset X_{T}$.  The \textbf{Abel map} $A_{q} \colon \Quot^{d}_{\omega} \to \bar{J}_{X}^{2g-2-d}$ out of the Quot scheme is defined by the rule
\begin{displaymath}
	[D_{\omega}] \mapsto [I_{D_{\omega}}],
\end{displaymath}
where $I_{D_{\omega}}$ is the kernel of the quotient map corresponding to $D_{\omega}$.
\end{definition}
These are the Abel maps that are constructed in \cite{altman80}, and they enjoy many of the properties that we would like the Abel map to have.  For example, if $D$ is an effective generalized divisor, then the fiber $A^{-1}([I_{D}])$ containing $[D]$ is equal to the complete linear system $|D|$, and the same remains true if we let $D$ be a generalized $\omega$-divisor and replace $\Hilb^{d}_{X}$ with $\Quot^{d}_{\omega}$.  

From this description of  $A$, we can conclude from Corollary~\ref{Cor: GorensteinGenLinSystem} that if $d \ge 2 g - 1$ and $X$ is Gorenstein, then the fibers of $A \colon \Hilb^{d}_{X} \to \bar{J}_{X}^{-d}$ are all $\bbP^{d-g}$'s, and in fact, this Abel map is smooth of  relative  dimension $d-g$ (\cite[Thm.~8.6]{altman80}).  Example~\ref{Example: LargeNonspecial} shows that this is not true when $X$ is non-Gorenstein, but if we replace $\Hilb^{d}_{X}$ with $\Quot^{d}_{X}$, then we recover the property that $A \colon \Quot^{d}_{\omega} \to \bar{J}_{X}^{-d}$ is smooth with $\bbP^{d-g}$ fibers (\cite[Thm.~8.4]{altman80}), as is suggested by Corollary~\ref{Cor: GenLinSystem}.

With this property of the Abel map, we can now construct the compactified Jacobian of a singular curve from the Quot scheme in the same manner that we constructed the Jacobian of a smooth curve from the symmetric power.  If we fix $d \ge 2g-1$, then linear equivalence partitions $\Quot^{d}_{\omega}$ into equivalence classes that are all isomorphic to $\bbP^{d-g}$.  In fact, linear equivalence is a smooth and projective equivalence relation on $\Quot^{d}_{\omega}$.  We can conclude that the quotient exists as a scheme, and this quotient is the compactified Jacobian $\bar{J}_{X}^{d}$. This argument is roughly the construction of  $\bar{J}^{d}_{X}$ given in \cite[Thm.~8.5]{altman80}.

The Abel map in Definition~\ref{Def: AbelForSing} however  suffers from one deficiency:  it does not quite resolve the indeterminacy of $A \colon X^{(d)} \dashrightarrow \bar{J}^{d}_{X}$.  The Abel map $A_{h} \colon \Hilb^{d}_{X} \to \bar{J}^{-d}_{X}$ sends a closed subscheme $D$ consisting of $d$ general points to the ideal sheaf $I_{D}$, but the composition $\Hilb^{d}_{X} \stackrel{\rho}{\longrightarrow} X^{(d)} \stackrel{A}{\dashrightarrow} \bar{J}^{d}_{X}$ sends $D$ to $\calL(D)$.  

This is not just an issue of convention.  In Example~\ref{Example: WrongDegree}, we saw an example of a generalized divisor $D$ of degree $d$ with the property that the degree of $\calL(D)$ is $d+1$.  Since the degree of a rank $1$, torsion-free sheaf is locally constant in flat families, we can conclude that the rule $D \mapsto \calL(D)$ does not define a morphism $\Hilb^{d}_{X} \to \bar{J}^{d}_{X}$ when $X$ is the curve from the example.

That curve is non-Gorenstein, and there are no problems provided we restrict our attention to Gorenstein curves.  Given $[D] \in \Hilb^{d}_{X}(T)$ (for some $T$), form the sheaf
\begin{displaymath}
	\calL(D) := \ShHom(I_{D}, \calO_{X_T}),
\end{displaymath}
If $X$ is Gorenstein, then for any rank $1$, torsion-free sheaf $I$, $\ShExt^{1}(I, \calO_{X})$ vanishes by \cite[6.1]{herzog71}.  In particular, this group vanishes when $I$ is the restriction of $I_{D}$ to a fiber of $X_{T} \to T$, so we can conclude by  \cite[Thm.~1.10]{altman80} that $\calL(D)$ is $T$-flat and its formation commutes with base-change.  (This is essentially \cite[2.2]{Presentation2}.) We can thus make the following definition.

\begin{definition} 
	Assume $X$ is Gorenstein.  Given a $k$-scheme $T$ and $[D] \in \Hilb^{d}_{X}(T)$,  set $\calL(D):= \ShHom(I_{D}, \calO_{X_{T}})$ where $I_{D}$ is the ideal of $D$.  The \textbf{modified Abel map} $A^{\vee}_{h} \colon \Hilb^{d}_{X} \to \bar{J}^{d}_{X}$ is defined by the rule $D \mapsto \calL(D)$.
\end{definition}
The modified Abel map is well-defined by the preceding discussion.  Furthermore, unlike $A_{h}$, the modified Abel map $A^{\vee}_{h}$ resolves the indeterminacy of $A \colon X^{(d)} \dashrightarrow \bar{J}^{d}_{X}$ in the sense described earlier.

When $X$ is non-Gorenstein, our discussion of generalized divisors from the previous section suggests that we should try to define a modified Abel map out of $\Quot^{d}_{\omega}$ rather than $\Hilb^{d}_{X}$.  We can always do this.  By \cite[2.2]{Presentation2} the sheaf 
\begin{displaymath}
	\calM(D_{\omega}) := \ShHom(I_{D_{\omega}}, \omega_{T})
\end{displaymath}
associated to $D_{\omega} \in \Quot^{d}_{\omega}(T)$ is always $T$-flat and its formation commutes with base-change.  We can therefore make the following definition.
\begin{definition}
	Given a $k$-scheme $T$ and $D_{\omega} \in \Quot^{d}_{\omega}(T)$,  set $\calM(D_{\omega}):= \ShHom(I_{D_{\omega}}, \omega_{T})$.  The \textbf{modified Abel map} $A_{q}^{\ast} \colon \Quot^{d}_{\omega} \to \bar{J}^{d}_{X}$ out of the Quot scheme is defined by the rule
\begin{displaymath}
	D_{\omega} \mapsto \calM(D_{\omega})
\end{displaymath}
\end{definition}
As with the modified Abel map out of the Hilbert scheme, this modification of the Abel map resolves the indeterminacy of $A \colon X^{(d)} \dashrightarrow \bar{J}^{d}_{X}$.

Having constructed Abel maps, we can now use these maps to define and study the theta divisor.  Imitating the construction for smooth curves, we make the following definition.
\begin{definition}
	The \textbf{theta divisor} $\Theta \subset \bar{J}^{g-1}_{X}$ of a curve $X$ is defined to be the image of $A_{q}^{\ast} \colon \Quot^{g-1}_{\omega} \to \bar{J}^{g-1}_{X}$ with the reduced scheme structure.
\end{definition}
We are abusing language in calling $\Theta$ a divisor because it is not known in general whether $\Theta$ is always a divisor.  The subscheme $\Theta$ is known to be a Cartier divisor when the singularities of $X$ are planar.  For such an $X$, Soucaris  \cite{soucaris} and Esteves \cite{esteves97} have proven that $\Theta$ is not only a Cartier divisor, but in fact $\Theta$ is an ample Cartier divisor.  They prove this statement by constructing and studying  $\Theta$ using the formalism of the determinant of cohomology.

Soucaris uses the same formalism to construct a subscheme in $\bar{J}^{g-1}_{X}$ when $X$ is a curve with arbitrary singularities, but then it is not known whether his subscheme coincides with $\Theta$, and it is not known whether his subscheme is a Cartier divisor.  By construction, Soucaris' subscheme is locally defined by a single equation, and Soucaris proves that the subscheme does not  contain an irreducible component of $\bar{J}_{X}^{g-1}$ \cite[Thm.~8]{soucaris}.  (His argument is essentially our Corollary~\ref{Cor: GenLinSystem}.)  This does not quite show that Soucaris' subscheme is a Cartier divisor.  To show that, it is necessary to also  prove that the subscheme does not contain any embedded component of $\bar{J}^{g-1}_{X}$, a result that is unknown.  More generally, it is not known whether $\bar{J}^{g-1}_{X}$ can have embedded components or even whether $\bar{J}^{g-1}_{X}$ can be non-reduced.

It is also not known whether Soucaris' subscheme equals $\Theta$ when $X$ has non-planar singularities.  The two subschemes are supported on the same subset of $\bar{J}^{g-1}_{X}$, so to show equality, it is enough to prove that Soucaris' subscheme is reduced.  

There is a third subscheme of $\bar{J}^{g-1}_{X}$ that we can define: the image of the Abel map $A_{h} \colon \Hilb^{g-1}_{X} \to \bar{J}^{-(g-1)}_{X}$.  This  subscheme can be identified with $\Theta$ when $X$ is Gorenstein, but we will see in Example~\ref{Example: AbelForSingularGenus2}  that the two subschemes can be different in general.

We conclude by describing the Abel map in some low genus examples. With the theory of the Abel map that we have developed, this will be a straightforward application of the results from the end of Section~\ref{Section: GeneralizedDivisors}.  As in Section~\ref{Sect: Jacobian}, we let $C^{1}_{d} \subset \Hilb^{d}_{X}$ denote the subset of effective generalized divisors $D$ with $\dim |D| \ge 1$ and similarly with $C^{1}_{d} \subset \Quot^{d}_{\omega}$.  

\begin{example}[Genus $1$]
If $X$ is  a genus $1$ curve, then the singularities of $X$ are planar and hence Gorenstein.  In Example~\ref{Example: SingGenus1}, we showed that $C^{1}_{g}=\emptyset$, so the Abel map $A_{h}^{\vee} \colon \Hilb^{1}_{X}=X \to \bar{J}^{1}_{X}$ is an isomorphism.  When $X$ is a nodal curve, this recovers the moduli-theoretic interpretation of the compactification of the generalized Jacobian from the beginning of Section~\ref{Sect: CompJac}.  The theta divisor $\Theta = \{ [\calO_{X}] \} \subset \bar{J}^{0}_{X}$ consists of a single point that lies in the smooth locus.
\end{example}

\begin{example}[Genus $2$] \label{Example: AbelForSingularGenus2}
In describing the structure of the Abel map, we need to distinguish between Gorenstein curves and non-Gorenstein curves.  When $X$ is a Gorenstein curve, we showed in Example~\ref{Example: SingGenus2} that $C^{1}_{g}$ is the rational curve $\bbP^{1} = C^{1}_{g} \subset \Hilb^{g}_{X}$ that consists of effective canonical divisors $K$.  The Abel map $A^{\vee}_{h}$ contracts this curve to the point $\{ [\omega] \} \subset J^{g}_{X}$ of the generalized Jacobian.  The locus $C^{1}_{g-1}$ is empty, so the theta divisor $\Theta \subset \bar{J}^{g-1}_{X}$ is an embedded copy of $\Hilb^{g-1}_{X}$, which is just $X$ itself.  Thus the structure of the Abel map is the same as for a smooth genus $2$ curve except the geometry of both the source and the target are more complicated.

What happens when $X$ is non-Gorenstein?  We will only consider the non-Gorenstein curve $X$ with a unique unibranched singularity $p_0 \in X$ from Example~\ref{Example: Dualizing}. For this curve, we can consider both the Abel map $A_{h}$ out of the Hilbert scheme and the Abel map $A_{q}^{\ast}$ out of the Quot scheme.  We begin by describing $A_{h}$.  The locus $C^{1}_{g} \subset \Hilb^{g}_{X}$ is the rational curve that parameterizes elements of $|p+p_0|$ (where $p \ne p_0$).  This curve is contracted by the Abel map $A_{h} \colon \Hilb^{g}_{X} \to \bar{J}^{-g}_{X}$, and away from this curve, $A_{h}$ is an isomorphism.  The locus $C^{1}_{g-1} \subset \Hilb^{1}_{X}$ is empty, so $X=\Hilb^{1}_{X} \to \bar{J}^{-(g-1)}_{X}$ is an embedding.

Let us now turn our attention to the Abel map $A^{\ast}_{q}$ out of the Quot scheme.  This Abel map contracts the rational curve $\bbP^{1} = C_{1}^{g} \subset \Quot^{g}_{\omega}$ parameterizing effective canonical $\omega$-divisors to the point $[\omega] \in \bar{J}^{2}_{X}$ and is an isomorphism away from $C^{1}_{g}$.  The locus $C^{1}_{g-1} \subset \Quot^{g-1}_{\omega}$ is empty, so $A^{\ast}_{q} \colon \Quot^{g-1}_{X} \to \bar{J}^{1}_{X}$ is an embedding.  How does $A_{q}^{\ast}$ compare with $A_{h}$?

In degree $g$, both Abel maps contract a rational curve $\bbP^{1}$, but the contractions are different. The image of $\bbP^{1}$ under $A_{h}$ is $[I_{p+p_0}] \in \bar{J}^{-g}_{X}$ which is a singularity, but  the image of $\bbP^{1}$ under $A_{q}^{\ast}$ is the smooth point  $[\omega]$ of $\bar{J}^{g}_{X}$.

To see that $[I_{p+p_0}]$ is a singularity, we estimate the tangent space dimension.  The tangent space  $\operatorname{T}_{[p+p_0]}(\Hilb^{2}_{X})$  is the sum of the tangent spaces $\operatorname{T}_{p}(X)$ and $\operatorname{T}_{p_0}(X)$, so $\dim \operatorname{T}_{[p+p_0]}(\Hilb^{2}_{X})=4$.  We can conclude that  $\dim \operatorname{T}_{[I_{p+p_0}]}(\bar{J}^{-g}_{X}) \ge 3$ because the  kernel of  
$$
	\operatorname{T}(A_{h}) \colon \operatorname{T}_{[p+p_0]}(\Hilb^{2}_{X}) \to \operatorname{T}_{[I_{p+p_{0}}]}(\bar{J}^{-g}_{X})
$$
is the $1$-dimensional   space 
$$
	\Hom(I_{p+p_0}, \calO_{X})/\Hom(I_{p+p_0}, I_{p+p_0}) = H^{0}(X, \calL(p+p_0))/H^{0}(X, \calO_{X}).
$$
Now the point $[I_{p+p_0}]$ lies in the closure of the line bundle locus, so either $[I_{p+p_0}]$ lies on the  intersection of two irreducible components or the local dimension of $\bar{J}^{-g}_{X}$ at $[I_{p+p_0}]$ is $2<3 \le \dim \operatorname{T}_{[I_{p+p_0}]}(\bar{J}^{-g}_{X})$.  In either case, we can conclude that $[I_{p+p_0}] \in \bar{J}^{-g}_{X}$ is a singularity.  

Tangent space techniques can also be used to prove that $[\omega] \in \bar{J}^{g}_{X}$ is a smooth point, and the reader is directed to the proof of \cite[Thm.~2.7]{kass:exoticcomp}  for the computation.  In \cite{kass:exoticcomp}, the author also  enumerates the irreducible components of $\bar{J}^{-g}_{X}$.  There are two.

The Abel maps in degree $g-1$ are also different.  Both $A_{h}$ and $A_{q}^{\ast}$ are embeddings of curves but of different curves. The morphism $A_{h}$ embeds $\Hilb^{g-1}_{X}$ which is isomorphic to the irreducible curve $X$.  The morphism $A_{q}^{\ast}$, on the other hand, embeds $\Quot^{g-1}_{X}$, and this scheme is not isomorphic to $X$.  As $g-1=1$, the Quot scheme $\Quot^{g-1}_{\omega}$ is isomorphic to $\bbP \omega$, the projectivization of $\omega$. From the presentation of $\omega$ in Example~\ref{Example: Dualizing} of Section~\ref{Section: Dualizing}, we see that $\Quot^{g-1}_{\omega}=\bbP \omega$ is a curve with two irreducible components, one  whose general element corresponds to a quotient supported at the singularity $p_0$ and one  whose general element corresponds to a quotient supported on the smooth locus $X_{\text{sm}}$.  The image $A_{q}^{\ast}(\Quot^{g-1}_{\omega})$ is the theta divisor $\Theta$, so this example shows that $\Theta$ and $A_{h}(\Hilb^{g-1}_{X})$ may not be isomorphic as schemes.  

\end{example}

\begin{example}[Genus $3$]
We only consider Gorenstein curves, and we consider the hyperelliptic curves and the non-hyperelliptic curves separately.  If $X$ is a singular hyperelliptic curve of genus $3$, then our description of complete linear systems on $X$  in Example~\ref{Example: SingGenus3} shows that $C^{1}_{g}$ is isomorphic to $X \times \bbP^1$, and the Abel map $A^{\vee}_{h} \colon C^{1}_{g} \to \bar{J}^{g}_{X}$ collapses the $\bbP^1$ factor.  In particular, the image $A_{h}^{\vee}(C^{1}_{g})$ is a copy of $X$. 

The locus $C^{1}_{g-1} \subset \Hilb^{g-1}_{X}$ is the rational curve that parameterizes effective canonical divisors.  As in the smooth case, this curve is blown down to a point $\{ [\omega] \} \subset \Theta$ that is a singularity by \cite[Prop.~6.1]{yano09}.  The point $[\omega]$ is the only point on the theta divisor with positive dimensional preimage under $A^{\vee}_{h}$, but it is not the only singularity. The theta divisor is also singular at points corresponding to sheaves that fail to be locally free.  This is another consequence of \cite[Prop.~6.1]{yano09}.

What if $X$ is non-hyperelliptic?  The canonical map then embeds $X$ as a degree $4$ plane curve, and we showed in Example~\ref{Example: SingGenus3} that the elements of $C^{1}_{g}$ are effective generalized divisors of degree $3$ that are contained in a line.  Arguing as in the smooth case (Example~\ref{Example: SmGenus3}), we can construct a map $C^{1}_{g} \to X$ with $\bbP^1$-fibers that are collapsed by $A^{\vee}_{h} \colon \Hilb^{g}_{X} \to \bar{J}^{g}_{X}$.  We can conclude that $A^{\vee}_{h}(C^{1}_{g})$ is a copy of $X$.  The locus $C^{1}_{g-1}$ is empty, so $\Theta$ is an embedded copy of $\Hilb^{2}_{X}$.
\end{example}

\begin{example}[Genus $4$]
Here we just describe the degree $g-1$ Abel map associated to the special singular curves that we studied in Example~\ref{Example: SingGenus4}.  Consider first the special non-hyperelliptic curve $X$ that lies on a singular quadric $Q \subset \bbP^{3}$. From our description in Example~\ref{Example: SingGenus4}, we see that $C^{1}_{g-1} \subset \Hilb^{g-1}_{X}$ is the rational curve that parameterizes effective degree $g-1$ generalized divisors that lie on a ruling of the quadric surface $Q \subset \bbP^3$.  This curve is contracted to a point by $A_{h}^{\vee}$.  Recall that every element of $C^{1}_{g-1}$ is a divisor that is not Cartier.  This is a new phenomenon: this is the first example of a point on the theta divisor that is not a line bundle and has positive dimensional preimage under $A^{\vee}_{h}$ .

The theta divisor of a singular hyperelliptic curve of genus $4$ contains similar points.  For such a curve, the work we did in Example~\ref{Example: SingGenus4} shows that $C^{1}_{g-1}$ is isomorphic to $X \times \bbP^1$, and the Abel map $A_{h}^{\vee}$ collapses the second $\bbP^1$ factor.  The image of $A^{\vee}_{h}$ is thus a copy of $X \subset \Theta$, and the singularities of $X$ correspond to points on the theta divisor that are not line bundles and have positive dimensional preimage.

In \cite{yano09}, the author and Casalaina-Martin computed the multiplicity of the theta divisor of a nodal curve at a point. In general, if $X$ is a nodal curve and $[I] \in \Theta$, then the main theorem of that paper asserts
\begin{displaymath}
	\operatorname{mult}_{I} \Theta = 2^{n} \cdot h^{0}(X,I),
\end{displaymath}
where $n$ is number of nodes at which $I$ fails to be locally free.  When $X$ is the special genus $4$ non-hyperelliptic curve and $[I] \in \Theta$ is the image of $C^{1}_{g-1}$, the theorem states that $[I] \in \Theta$ is a multiplicity $4$ point.
\end{example}

\section{Appendix: The dualizing sheaf and coherent duality} \label{Section: Dualizing}
The dualizing sheaf $\omega$ of a curve plays an important role in the study of compactified Jacobians and generalized divisors.  Here we recall the definition and the basic properties of $\omega$.  Two references for this material are \cite{altman70} and \cite[Chap.~III, Sect.~7]{hartshorne77}.

The dualizing sheaf $\omega$ of a curve $X$ is defined as follows.  Given a point $p \in X^{\nu}$ of the normalization of $X$, fix a uniformizer $t$ of $X^{\nu}$ at $p$.  We define the \textbf{residue} $\operatorname{Res}_{p}(\eta)$ of a rational $1$-form $\eta$ at $p$ as follows.  In the local ring $\calO_{X^{\nu}, p}$, we can write
$\eta$ as
\begin{displaymath}
	\eta = (b(t) + a_{-1} t^{-1} + a_{-2} t^{-2} + \dots) dt
\end{displaymath}
for $b(t) \in \calO_{X^{\nu}, p}$ and $a_{n} \in k$.
We define
\begin{displaymath}
	\operatorname{Res}_{p}(\eta) := a_{-1}.
\end{displaymath}
The functional $\operatorname{Res}_{p}$ is independent of the choice of $t$ (though the proof is non-trivial when $\operatorname{char}(k)>0$).  The dualizing sheaf can be defined in terms of residues.
\begin{definition} \label{Definition: DualizingMod}
	The \textbf{dualizing sheaf} $\omega$ of the curve $X$ is defined to be the sheaf whose sections $\eta \in H^{0}(U, \omega)$ over an open subset $U \subset X$ are rational $1$-forms $\eta$ with the property that
	\begin{displaymath}
		\sum_{\nu(p)=q} \operatorname{Res}_{p}(f \eta) = 0
	\end{displaymath}
	for all regular functions $ f \in H^{0}(U, \calO_{X})$ and all points $q \in U$.
\end{definition}

The dualizing module admits a distinguished functional $t \colon H^{1}(X, \omega) \to k$ (whose definition we omit) that induces an isomorphism 
\begin{displaymath}
	\Ext^{n}(F, \omega) \cong H^{1-n}(X, F)^{\vee}
\end{displaymath}
for every coherent sheaf $F$ and every integer $n$.  This statement is the Coherent Duality Theorem (see \cite[Chap.~IV, Sect.~5]{altman70} or \cite[Thm.~7.6]{hartshorne77}).    By general formalism, the pair $(\omega, t)$ is unique up to a unique isomorphism.

We can say even more when $F$ is a rank $1$, torsion-free sheaf.  If $F=I$ is a rank $1$, torsion-free sheaf, then the natural map $H^{n}(X, \ShHom(I, \omega)) \to \Ext^{n}( I, \omega)$ is an isomorphism for all $n$ because the higher cohomology sheaves $\ShExt^{n}(I, \omega)$, $n>0$, vanish  \cite[6.1]{herzog71}.  In particular, if $I=\calM(D_{\omega})$ is the rank $1$, torsion-free sheaf associated to a generalized $\omega$-divisor, then coherent duality takes the form 
\begin{equation} \label{OmegaSerreDuality}
	H^{n}(X, \calM(\operatorname{adj} D_{\omega}))  \cong H^{1-n}(X, \calM(D_{\omega}))^{\vee}
\end{equation}
Similarly, if $X$ is Gorenstein (so the adjoint of a generalized divisor exists), then for every generalized divisor $D$, we have
\begin{equation} \label{SerreDuality}
	H^{n}(X, \calL(\operatorname{adj} D)) \cong H^{1-n}(X, \calL(D))^{\vee}.
\end{equation}

We conclude this section by computing the dualizing sheaf of a specific non-Gorenstein curve.
\begin{example} \label{Example: Dualizing}
	Define $X$ to be the curve constructed by gluing the affine curves 
	\begin{align*}
	X_1 :=& \Spec(k[t^3,t^4,t^5]), \\ 
	X_2 :=& \Spec(k[s])
\end{align*}
by the isomorphism $t = s^{-1}$.  This is a rational curve of genus $2$ with a unique singularity that is non-Gorenstein and unibranched.

Using Definition~\ref{Definition: DualizingMod}, we see that the dualizing sheaf $\omega$
is the subsheaf of the sheaf $\omega_{\calK}$ of rational $1$-forms generated by
\begin{gather} \label{Eqn: DualizingExample}
	dt/t^3, dt/t^2 \text{ on $X_1$}; \\
	ds \text{ on $X_2$.}\notag
\end{gather}
\end{example}

\section{Acknowledgments}
The author thanks the anonymous referee and Maxim Arap for helpful comments on earlier drafts of this article.  The author also thanks  David Rydh for a helpful email exchange about the Hilbert--Chow morphism.


\bibliographystyle{amsalpha}
\bibliography{bibl}

\providecommand{\bysame}{\leavevmode\hbox to3em{\hrulefill}\thinspace}
\providecommand{\MR}{\relax\ifhmode\unskip\space\fi MR }
\providecommand{\MRhref}[2]{%
  \href{http://www.ams.org/mathscinet-getitem?mr=#1}{#2}
}
\providecommand{\href}[2]{#2}
\begin{thebibliography}{ACGH85}

\bibitem[ACGH85]{ACGH}
E.~Arbarello, M.~Cornalba, P.~A. Griffiths, and J.~Harris, \emph{Geometry of
  algebraic curves. {V}ol. {I}}, Grundlehren der Mathematischen Wissenschaften
  [Fundamental Principles of Mathematical Sciences], vol. 267, Springer-Verlag,
  New York, 1985.

\bibitem[AIK77]{Irred1977}
Allen~B. Altman, Anthony Iarrobino, and Steven~L. Kleiman, \emph{Irreducibility
  of the compactified {J}acobian}, pp.~1--12, Sijthoff and Noordhoff, Alphen
  aan den Rijn, 1977.

\bibitem[AK70]{altman70}
Allen Altman and Steven Kleiman, \emph{Introduction to {G}rothendieck duality
  theory}, Lecture Notes in Mathematics, Vol. 146, Springer-Verlag, Berlin,
  1970. \MR{0274461 (43 \#224)}

\bibitem[AK80]{altman80}
Allen~B. Altman and Steven~L. Kleiman, \emph{Compactifying the {P}icard
  scheme}, Adv. in Math. \textbf{35} (1980), no.~1, 50--112.

\bibitem[Art69]{artin69}
M.~Artin, \emph{Algebraization of formal moduli. {I}}, Global {A}nalysis
  ({P}apers in {H}onor of {K}. {K}odaira), Univ. Tokyo Press, Tokyo, 1969,
  pp.~21--71. \MR{0260746 (41 \#5369)}

\bibitem[Art74]{artin74}
\bysame, \emph{Versal deformations and algebraic stacks}, Invent. Math.
  \textbf{27} (1974), 165--189. \MR{0399094 (53 \#2945)}

\bibitem[BLR90]{bosch:1990}
Siegfried Bosch, Werner L{\"u}tkebohmert, and Michel Raynaud, \emph{N\'eron
  models}, Ergebnisse der Mathematik und ihrer Grenzgebiete (3) [Results in
  Mathematics and Related Areas (3)], vol.~21, pp.~x+325, Springer-Verlag,
  Berlin, 1990.

\bibitem[CCE08]{coelho08}
Lucia Caporaso, Juliana Coelho, and Eduardo Esteves, \emph{Abel maps of
  {G}orenstein curves}, Rend. Circ. Mat. Palermo (2) \textbf{57} (2008), no.~1,
  33--59. \MR{2420522 (2009e:14045)}

\bibitem[CE07]{caporaso07}
Lucia Caporaso and Eduardo Esteves, \emph{On {A}bel maps of stable curves},
  Michigan Math. J. \textbf{55} (2007), no.~3, 575--607. \MR{2372617
  (2009a:14007)}

\bibitem[CMK12]{yano09}
Sebastian Casalaina-Martin and Jesse~Leo Kass, \emph{A {R}iemann singularity
  theorem for integral curves}, Amer. J. Math. \textbf{134} (2012), no.~5,
  1143--1165.

\bibitem[CP10]{coelho10}
Juliana Coelho and Marco Pacini, \emph{Abel maps for curves of compact type},
  J. Pure Appl. Algebra \textbf{214} (2010), no.~8, 1319--1333. \MR{2593665
  (2011g:14071)}

\bibitem[EGK00]{Presentation2}
Eduardo Esteves, Mathieu Gagn{\'e}, and Steven Kleiman, \emph{Abel maps and
  presentation schemes}, Comm. Algebra \textbf{28} (2000), no.~12, 5961--5992,
  Special issue in honor of Robin Hartshorne.

\bibitem[Est97]{esteves97}
Eduardo Esteves, \emph{Very ampleness for theta on the compactified
  {J}acobian}, Math. Z. \textbf{226} (1997), no.~2, 181--191.

\bibitem[FG05]{fantechi05}
Barbara Fantechi and Lothar G{\"o}ttsche, \emph{Local properties and {H}ilbert
  schemes of points}, Fundamental algebraic geometry, Math. Surveys Monogr.,
  vol. 123, Amer. Math. Soc., Providence, RI, 2005, pp.~139--178. \MR{2223408}

\bibitem[Har77]{hartshorne77}
Robin Hartshorne, \emph{Algebraic geometry}, Springer-Verlag, New York, 1977,
  Graduate Texts in Mathematics, No. 52. \MR{0463157 (57 \#3116)}

\bibitem[Har86]{hartshorne86}
\bysame, \emph{Generalized divisors on {G}orenstein curves and a theorem of
  {N}oether}, J. Math. Kyoto Univ. \textbf{26} (1986), no.~3, 375--386.
  \MR{857224 (87k:14036)}

\bibitem[Har94]{hartshorne94}
\bysame, \emph{Generalized divisors on {G}orenstein schemes}, Proceedings of
  {C}onference on {A}lgebraic {G}eometry and {R}ing {T}heory in honor of
  {M}ichael {A}rtin, {P}art {III} ({A}ntwerp, 1992), vol.~8, 1994,
  pp.~287--339. \MR{1291023 (95k:14008)}

\bibitem[Har07]{hartshorne07}
\bysame, \emph{Generalized divisors and biliaison}, Illinois J. Math.
  \textbf{51} (2007), no.~1, 83--98 (electronic). \MR{2346188 (2008j:14010)}

\bibitem[HK71]{herzog71}
J{\"u}rgen Herzog and Ernst Kunz (eds.), \emph{Der kanonische {M}odul eines
  {C}ohen-{M}acaulay-{R}ings}, Lecture Notes in Mathematics, Vol. 238,
  Springer-Verlag, Berlin, 1971, Seminar {\"u}ber die lokale Kohomologietheorie
  von Grothendieck, Universit{\"a}t Regensburg, Wintersemester 1970/1971.
  \MR{0412177 (54 \#304)}

\bibitem[Ive70]{iversen70}
Birger Iversen, \emph{Linear determinants with applications to the {P}icard
  scheme of a family of algebraic curves}, Lecture Notes in Mathematics, Vol.
  174, Springer-Verlag, Berlin, 1970. \MR{0292835 (45 \#1917)}

\bibitem[Kas12]{kass:exoticcomp}
Jesse~Leo Kass, \emph{An explicit non-smoothable component of the compactified
  {J}acobian}, J. Algebra \textbf{370} (2012), 326--343. \MR{2966842}

\bibitem[KK81]{Kleppe81}
Hans Kleppe and Steven~L. Kleiman, \emph{Reducibility of the compactified
  {J}acobian}, Compositio Math. \textbf{43} (1981), no.~2, 277--280.

\bibitem[Kle66]{kleiman66}
Steven~L. Kleiman, \emph{Toward a numerical theory of ampleness}, Ann. of Math.
  (2) \textbf{84} (1966), 293--344. \MR{0206009 (34 \#5834)}

\bibitem[Kle05]{kleiman05}
\bysame, \emph{The {P}icard scheme}, Fundamental algebraic geometry, Math.
  Surveys Monogr., vol. 123, Amer. Math. Soc., Providence, RI, 2005,
  pp.~235--321.

\bibitem[KM09]{kleiman09}
Steven~Lawrence Kleiman and Renato~Vidal Martins, \emph{The canonical model of
  a singular curve}, Geom. Dedicata \textbf{139} (2009), 139--166. \MR{2481842
  (2010d:14038)}

\bibitem[Mum70]{mumford70}
David Mumford, \emph{Abelian varieties}, Tata Institute of Fundamental Research
  Studies in Mathematics, No. 5, Published for the Tata Institute of
  Fundamental Research, Bombay, 1970. \MR{0282985 (44 \#219)}

\bibitem[Mum75]{mumford75}
\bysame, \emph{Curves and their {J}acobians}, The University of Michigan Press,
  Ann Arbor, Mich., 1975. \MR{0419430 (54 \#7451)}

\bibitem[Ryd08a]{rydh08a}
David Rydh, \emph{Families of zero cycles and divided powers: I.
  representability}, arXiv:0803.0618v1, 2008.

\bibitem[Ryd08b]{rydh08b}
\bysame, \emph{Families of zero cycles and divided powers: I\uppercase{I}. the
  universal family}, \url{http://www.math.kth.se/~dary/}, 2008.

\bibitem[SGA73]{SGA4}
\emph{Th\'eorie des topos et cohomologie \'etale des sch\'emas. {T}ome 3},
  Lecture Notes in Mathematics, Vol. 305, Springer-Verlag, Berlin, 1973,
  S{\'e}minaire de G{\'e}om{\'e}trie Alg{\'e}brique du Bois-Marie 1963--1964
  (SGA 4), Dirig{\'e} par M. Artin, A. Grothendieck et J. L. Verdier. Avec la
  collaboration de P. Deligne et B. Saint-Donat. \MR{0354654 (50 \#7132)}

\bibitem[Sou94]{soucaris}
A.~Soucaris, \emph{The ampleness of the theta divisor on the compactified
  {J}acobian of a proper and integral curve}, Compositio Math. \textbf{93}
  (1994), no.~3, 231--242.

\end{thebibliography}
\end{document}